\documentclass{amsart}
\usepackage{amssymb,latexsym}
\usepackage{amscd,amsthm}

\usepackage{tikz-cd}

\usepackage[all]{xy}

\newtheorem{theorem}{Theorem}[section]
\newtheorem{lemma}[theorem]{Lemma}

\newtheorem{proposition}[theorem]{Proposition}
\newtheorem{corollary}[theorem]{Corollary}

\theoremstyle{definition}

\newtheorem*{remark}{Remark}

\newenvironment{theorembis}[1]
  {\renewcommand{\thetheorem}{\ref{#1}$'$}%
   \addtocounter{theorem}{-1}%
   \begin{theorem}}
  {\end{theorem}}

\DeclareMathOperator{\Ext}{Ext}

\DeclareMathOperator{\Hom}{Hom}

\DeclareMathOperator{\cok}{cok}

\newcommand{\cat}[1]{\mathcal{#1}}           

\newcommand{\tensor}{\otimes}

\newcommand{\class}[1]{\mathcal{#1}}   

\newcommand{\Z}{\mathbb{Z}}

\newcommand{\mathcolon}{\colon\,} 

\newcommand{\ch}{\textnormal{Ch}(R)}
\newcommand{\cha}[1]{\textnormal{Ch}(\mathcal{#1})}

\newcommand{\qcox}{\textnormal{Qco}(\mathbb{X})}
\newcommand{\chqcox}{\textnormal{Ch}(\mathbb{X})}

\newcommand{\sheafhom}{\mathcal{H}\mspace{-2mu}\mathit{o}\mathit{m}\mspace{1mu}}  

\newcommand{\tilclass}[1]{\widetilde{\class{#1}}}
\newcommand{\dgclass}[1]{dg\widetilde{\class{#1}}}

\newcommand{\dwclass}[1]{dw\widetilde{\class{#1}}}
\newcommand{\exclass}[1]{ex\widetilde{\class{#1}}}
\newcommand{\toclass}[1]{to\widetilde{\class{#1}}}

\newcommand{\rightperp}[1]{#1^{\perp}}
\newcommand{\leftperp}[1]{{}^\perp #1}

\newcommand{\homcomplex}{\mathit{Hom}}

\begin{document}

\title[Quillen equivalences inducing Grothendieck]{Quillen equivalences inducing Grothendieck duality for unbounded chain complexes of sheaves}

\author{Sergio Estrada}
 \address{S.E. \ Universidad de
  Murcia\\ Facultad de Matemáticas \\ Campus de Espinardo \\ Murcia 30100, Spain} \email[Sergio Estrada]{sestrada@um.es}
\urladdr{https://webs.um.es/sestrada/}
\author{James Gillespie}
\address{J.G. \ Ramapo College of New Jersey \\
         School of Theoretical and Applied Science \\
         505 Ramapo Valley Road \\
         Mahwah, NJ 07430\\ U.S.A.}
\email[Jim Gillespie]{jgillesp@ramapo.edu}
\urladdr{http://pages.ramapo.edu/~jgillesp/}
\thanks{S.E.\ was supported by grant PID2020-113206GB-I00/AEI/10.13039/501100011033.}

\date{\today}



\begin{abstract}
Let $\mathbb{X}$ be a semiseparated Noetherian scheme with a dualizing complex $D$. We lift some well-known triangulated equivalences associated with Grothendieck duality to Quillen equivalences of model categories. In the process we are able to show that the Gorenstein flat model structure, on the category of quasi-coherent sheaves on $\mathbb{X}$, is Quillen equivalent to the Gorenstein injective model structure.
Also noteworthy is that we extend the recollement of Krause to hold without the Noetherian condition. Using a set of flat generators, it holds for any quasi-compact semiseparated scheme $\mathbb{X}$. With this we also show that the Gorenstein injective quasi-coherent sheaves are the fibrant objects of a cofibrantly generated abelian model structure for any semiseparated Noetherian scheme $\mathbb{X}$. Finally, we consider both the injective and (mock) projective approach to Tate cohomology of quasi-coherent sheaves. They agree whenever $\mathbb{X}$ is a semiseparated Gorenstein scheme of finite Krull dimension.
\end{abstract}

\maketitle

\section{Introduction}

The central theme of this article is lifting some well-known triangulated equivalences to Quillen equivalences.
From a homotopy theoretic standpoint it is desirable to have a stable model structure on the  ground category, $\qcox$ or $\chqcox$, which possesses all small limits and colimits. The Quillen model structures and equivalences we construct are ultimately an expression of Grothendieck duality.
Theorem~\ref{them-Quillen-Krause-Iyengar-Murfet} is the Quillen equivalence version of Murfet's main theorem~\cite[Theorem~8.4]{murfet-thesis} from his thesis. Hence the title of this paper. It might go without saying, and it has been explained well by Murfet in his thesis and by other authors, but the extension of Grothendieck duality is the culmination of many interesting articles. Most notably,  \cite{krause-stable derived cat of a Noetherian scheme,jorgensen-homotopy-projectives, iyengar-krause,neeman-flat, murfet-thesis, murfet-salarian}.
The current work depends on these articles, but we also contribute several results of interest that go beyond the construction of model structures and Quillen equivalences.
We explain each below, but first we wish to point out the impetus for all of this.

We recall that, for modules over a ring $R$, a cotorsion module $C$ is one such that $\Ext^1_R(F,C) = 0$ for every flat module $F$.
An important recent result is that for chain complexes of modules over $R$, every chain map $f : F \xrightarrow{} C$ is null homotopic whenever $F$ is an exact complex with flat cycles and $C$ is any complex of cotorsion $R$-modules. This appeared in~\cite[Theorem~5.3]{bazzoni-cortes-estrada}. A precursor to this was the surprising (though weaker) fact proved by {\v{S}}\v{t}ov{\'{\i}}{\v{c}}ek that every exact complex of injective $R$-modules has cotorsion cycles~\cite{stovicek-purity}.  We can ask the same questions for (quasi-coherent) sheaves on $\mathbb{X}$. The following result gives an affirmative answer to this and allows us to achieve the most interesting results in the present article.

\begin{theorem}\cite[Theorem~3.3]{cet-G-flat-stable-scheme}\label{them-cotorsion-complexes-scheme}
 Assume $\mathbb{X}$ is a semiseparated and quasi-compact scheme. Let $C$ be a chain complex of cotorsion (quasi-coherent) sheaves. Then every chain complex $f : F \xrightarrow{} C$ from  an exact complex  $F$ with flat cycles is null-homotopic. And, if $C$ is exact then it has cotorsion cycles.
\end{theorem}

We will refer to this result often throughout the present article.
In the notation of the second author, from~\cite{gillespie,gillespie-degreewise-model-strucs}, it means $\dgclass{C} = \dwclass{C}$, and $\tilclass{C} = \exclass{C}$. So we will make these substitutions freely throughout the paper. The first noteworthy comment is that this simplifies the descriptions of the homotopy categories studied by Murfet. For examples, see Theorem~\ref{them-contra-model-scheme} and the diagram in Proposition~\ref{prop-units-counits-weak-equivalences}.

Now let us explain in different sections below what these realizations help us to achieve in the present article. It will help the reader to keep in mind that the scheme $\mathbb{X}$ has fewer hypotheses in the beginning of the paper, and the stronger hypotheses (Noetherian, dualizing complex, Gorenstein) are added as the paper progresses.


\subsection{Krause's stable derived category and recollement}
Now we discuss the results of Section~\ref{sec-recollement}.
Becker showed in his article~\cite{becker}  that Krause's recollement $S(Inj) \xrightarrow{} K(Inj) \xrightarrow{} \class{D}(R)$, from~\cite{krause-stable derived cat of a Noetherian scheme}, holds for any ring $R$, even without the Noetherian hypothesis. Here $K(Inj)$ is the homotopy category of all chain complexes of injectives and $S(Inj)$ is Krause's \emph{stable derived category}. It is the full subcategory of $K(Inj)$ consisting of all exact (acyclic) complexes of injectives.  Krause's original work in~\cite{krause-stable derived cat of a Noetherian scheme} was in the setting of separated Noetherian schemes and locally noetherian categories. In Neeman's paper~\cite{neeman-homotopy category of injectives}, the question is raised as to when the recollement holds.
In that article, he also studies the homotopy category of injectives and  produces an explicit example of a locally noetherian Grothendieck category $\cat{G}$ in which products of exact complexes of injectives need not be exact. It implies the failure of recollement, and by~\cite{krause-stable derived cat of a Noetherian scheme}, $\class{D}(\class{G})$ is not compactly generated.

In general, the recollement appears to be connected to the existence of a nice set of generators.  For example, $R$ is a projective generator for $R$-Mod, which helps Becker get his result in~\cite{becker}. It is shown in~\cite{gillespie-models-for-hocats-of-injectives} that the recollement holds when $\cat{G}$ comes with a set of generators of finite projective dimension.
But using Theorem~\ref{them-cotorsion-complexes-scheme} we can now use flat generators for $\qcox$ to show that the recollement holds for any semiseparated and quasi-compact scheme $\mathbb{X}$.  See Corollary~\ref{cor-recollement of krause}. To do this, we first construct, see Theorem~\ref{them-exact Inj model structure}, an injective model structure for Krause's stable derived category.
 Another point of interest here is that any quasi-compact and \emph{quasi-separated} scheme possessing a generating set of flat (quasi-coherent) sheaves is necessarily  semiseparated, by~\cite{slavick-stovicek-flat-generators}. So with our approach of using flat generators, this is the best result we can expect. After all, if one expects $S(Inj)$ to be the homotopy category of some abelian model structure then it is necessary for there to be some set of generators in the left side of the cotorsion pair.

\subsection{Totally acyclic complexes of injectives and Gorenstein injectives}
Let us point out here the highlights of Section~\ref{sec-totally-acyclic-inj}. We now assume $\mathbb{X}$ is a semiseparated Noetherian scheme.
Building on the results of Section~\ref{sec-recollement},  we are able to construct in Theorem~\ref{them-totally Inj model structure} a cofibrantly generated model structure on $\chqcox$ whose homotopy category recovers the full subcategory of $S(Inj)$ consisting of the totally acyclic complexes of injective (quasi-coherent) sheaves.  Moreover, we see in Theorem~\ref{them-G-inj} that the Gorenstein injectives are the fibrant objects of a cofibrantly generated model structure on $\qcox$, and that it is Quillen equivalent to the one on $\chqcox$. The root of these ideas also goes back to~\cite[Theorem~7.12]{krause-stable derived cat of a Noetherian scheme}. We are now able to show that the homotopy categories expand to include all objects of $\qcox$ and $\chqcox$, these models are cofibrantly generated, and lastly they are connected by a Quillen equivalence.
Keeping with the theme of this paper, we note that this is a model category extension of Krause's result~\cite[Proposition~7.2]{krause-stable derived cat of a Noetherian scheme}.

Our proof methods in this section trace back to~\cite{bravo-gillespie-hovey}. That approach was generalized to noetherian categories in~\cite{gillespie-models-for-hocats-of-injectives}, but again it was assumed there that the category possessed a set of generators of finite projective dimension. These may now be replaced by a set of flat generators. So again, it is $\dgclass{C} = \dwclass{C}$ (Theorem~\ref{them-cotorsion-complexes-scheme}) that makes our proofs of Theorems~\ref{them-totally Inj model structure} and Theorem~\ref{them-G-inj} possible.

\subsection{F-totally acyclic complexes of flats and Gorenstein flats} In section~\ref{sec-F-totally} we look at the dual notion of the Gorenstein flat sheaves. Again, assuming $\mathbb{X}$ is a semiseparated Noetherian scheme, we show that the Gorenstein flat model structure on $\qcox$, from~\cite{cet-G-flat-stable-scheme}, is Quillen equivalent to the F-totally acyclic model structure on $\chqcox$, from~\cite{estrada-gillespie-coherent-schemes}. See Theorem~\ref{them-G-Quillen-equiv}. This solves an implicit question left open by~\cite[Corollary~4.6]{cet-G-flat-stable-scheme} where it was shown that there is indeed a triangulated equivalence between the two at the level of the homotopy category of cofibrant-fibrant subobjects of the two categories.
The authors are quite pleased with this result, as a proof avoiding projectives eluded us while writing~\cite{estrada-gillespie-coherent-schemes}. A careful look at the proof reveals that again $\dgclass{C} = \dwclass{C}$ (Theorem~\ref{them-cotorsion-complexes-scheme}) is a key ingredient to the proof of Theorem~\ref{them-G-Quillen-equiv}.


\subsection{Equivalence of the Gorenstein flat and Gorenstein injective model structures}
It is in Section~\ref{sec-Quillen-dualizing} that we work out the details showing that a dualizing complex $D$ induces a Quillen equivalence from the contraderived category of $\mathbb{X}$ to the coderived category of $\mathbb{X}$. Section~\ref{sec-G-equivalence} continues in this direction:  Murfet showed in his thesis that a complex $F$ of flat sheaves is F-totally acyclic if and only if $D \tensor F$ is a totally acyclic complex of injective sheaves. This allows us to show that the Gorenstein flat model structure is Quillen equivalent to the Gorenstein injective model structure whenever $\mathbb{X}$ is a semiseparated Noetherian scheme with a dualizing complex $D$. See Theorem~\ref{them-Murfet-Salarian} and Theorem~\ref{them-G-flat-G-inj}.

Other authors have studied when the stable category of Gorenstein projectives is equivalent to the stable category of Gorenstein injectives.
For example, some Quillen equivalences in this direction were found in~\cite{DEH-stable-categories}.
Also, it is shown in~\cite{zheng-huang} that, for a Noetherian ring $R$ with a dualizing complex,  the stable category of Gorenstein projective $R$-modules is equivalent to the stable category of Gorenstein injective $R$-modules. Our work in Section~\ref{sec-G-equivalence} extends this to non-affine schemes.

\subsection{Projective and injective Tate cohomology for schemes}

In~\cite{gillespie-canonical resolutions} we find what essentially amounts to a theory of Tate cohomology that can be applied to any hereditary abelian model structure. So with the model category approach taken in this article, it is natural for us to consider what happens when we apply this theory to the Gorenstein flat and Gorenstein injective model structures on $\qcox$. This topic is taken up in the final section of the paper.
We recover Krause's definition of injective Tate cohomology from~\cite{krause-stable derived cat of a Noetherian scheme}.  Using the Gorenstein flat model structure we define (mock) projective Tate cohomology, which is closely related to the one defined in~\cite{asad-salarian}. The main properties of our projective Tate cohomology are listed in Proposition~\ref{theorem-projective-Tate}. Finally, we show at the end of the paper, see Theorem~\ref{them-mock-tate-inj-tate}, that the two cohomology theories agree when $\mathbb{X}$ is a semiseparated Gorenstein scheme of finite Krull dimension.

\section{Preliminaries}

Throughout this paper $\mathbb{X}$ will always be, at the very least, a semiseparated and quasi-compact scheme. We will specify when we add hypotheses; for most of the paper $\mathbb{X}$ will be semiseparated Noetherian. The category of all quasi-coherent sheaves on $\mathbb{X}$ will be denoted $\qcox$. It is a Grothendieck abelian category.

\emph{Again, throughout the paper, all sheaves are assumed to be quasi-coherent over $\mathbb{X}$. In particular, unless stated otherwise, the word ``sheaf'' is to be interpreted as ``quasi-coherent sheaf''.}

For sheaves $F$ and $G$, we let $\Hom_{\mathbb{X}}(F,G)$ denote the categorical Hom, with values in abelian groups. Similarly, $\Ext^n_{\mathbb{X}}(F,G)$ will denote the Yoneda Ext groups. We will sometimes discuss the affine case in the context of a (commutative) ring with identity, $R$. We use the notation $\Hom_R(M,N)$, and $\Ext^n_R(M,N)$, etc.

\subsection{Chain complexes of (quasi-coherent) sheaves}\label{section-homcomplex-schemes}
The category of all chain complexes of sheaves will be denoted $\chqcox$, and its chain homotopy category $K(\mathbb{X})$.
Our convention is that the differential lowers degree, so $$X \equiv \cdots
\xrightarrow{} X_{n+1} \xrightarrow{d_{n+1}} X_{n} \xrightarrow{d_n}
X_{n-1} \xrightarrow{} \cdots$$ is a chain complex.
The
\emph{suspension of $X$}, denoted $\Sigma X$, is the complex given by
$(\Sigma X)_{n} = X_{n-1}$ and $(d_{\Sigma X})_{n} = -d_{n}$.  The
complex $\Sigma (\Sigma X)$ is denoted $\Sigma^{2} X$ and inductively
we define $\Sigma^{n} X$ for all positive integers. We also set $\Sigma^0 X = X$ and define $\Sigma^{-1}$ by shifting indices in the other direction.
We call a chain complex $X$ \emph{exact} or \emph{acyclic} (we use these terms interchangeably) if $H_nX = 0$ for all $n$.

The $\chqcox$ analogs of the groups $\Hom_{\mathbb{X}}(F,G)$ and $\Ext^n_{\mathbb{X}}(F,G)$ will be denoted $\Hom_{\chqcox}(X,Y)$ and $\Ext_{\chqcox}^n(X,Y)$, or even simply by $\Hom(X,Y)$ and $\Ext^n(X,Y)$. Again, the values of these functors are abelian groups. We also define, in the usual way for any complete abelian category, $\homcomplex_{\mathbb{X}}(X,Y)$ to
be the complex of abelian groups $$ \cdots \xrightarrow{} \prod_{k \in
\Z} \Hom_{\mathbb{X}}(X_{k},Y_{k+n}) \xrightarrow{\delta_{n}} \prod_{k \in \Z}
\Hom_{\mathbb{X}}(X_{k},Y_{k+n-1}) \xrightarrow{} \cdots ,$$ where $(\delta_{n}f)_{k}
= d_{k+n}f_{k} - (-1)^n f_{k-1}d_{k}$.
This gives a left exact functor
$$\homcomplex_{\mathbb{X}}(X,-) \mathcolon \chqcox \xrightarrow{} \textnormal{Ch}(\Z)$$
and similarly the contravariant functor $\homcomplex_{\mathbb{X}}(-,Y)$ sends right
exact sequences to left exact sequences. It will be important for us to note that it is exact when each $Y_{n}$ is an injective
sheaf. Also important to us is the fact the homology satisfies $$H_n[\homcomplex_{\mathbb{X}}(X,Y)] = K(\mathbb{X})(X,\Sigma^{-n} Y).$$ This can be verified immediately.

Finally, the Yoneda Ext functor $\Ext^1_{\chqcox}$ has a subfunctor which we will denote by $\Ext^1_{dw}$, consisting of all degreewise split short exact sequences. The following lemma gives a well-known connection, true for general abelian categories, between $\Ext^1_{dw}$ and the hom-complex $\homcomplex_{\mathbb{X}}$.
\begin{lemma}\label{lemma-homcomplex-basic-lemma}
For chain complexes $X$ and $Y$, we have isomorphisms:
$$\Ext^1_{dw}(X,\Sigma^{(-n-1)}Y) \cong H_n \homcomplex_{\mathbb{X}}(X,Y) =
K(\mathbb{X})(X,\Sigma^{-n} Y)$$ In particular, for chain complexes $X$ and $Y$, $\homcomplex_{\mathbb{X}}(X,Y)$ is
exact iff for any $n \in \Z$, any chain map $f \mathcolon \Sigma^nX \xrightarrow{} Y$ is
homotopic to 0 (or iff any chain map $f \mathcolon X \xrightarrow{} \Sigma^nY$ is homotopic
to 0).
\end{lemma}

\subsection{Closed symmetric monoidal structures}\label{section-csms}

The category $\qcox$ is a closed symmetric monoidal category with respect to the usual tensor product $F \otimes G$ of sheaves. The internal hom, which we will denote by $\sheafhom_{qc}(F,G)$, is more subtle; there is a brief reminder in~\cite[\S2.1]{murfet-salarian} and~\cite[beginning of \S6]{gillespie-quasi-coherent}. In a general way, the closed symmetric monoidal structure lifts to one on $\chqcox$; for example, see~\cite[Prop.~9.2.1]{hovey-axiomatic stable homotopy}. In particular, for any chain complex $X$ we have a functor $$- \tensor X : \chqcox \xrightarrow{} \chqcox$$
and it has a right adjoint that we again will denote by
$$\sheafhom_{qc}(X,-) : \chqcox \xrightarrow{} \chqcox.$$
We just recall that $X \tensor Y$ is defined so that in degree $n$ we have a coproduct
\begin{equation}\label{equation-tensor}
(X \tensor Y)_n = \bigoplus_{i+j=n} X_i
\tensor Y_j
\end{equation}
 and the differential uses the sign trick.
On the other hand, the closed structure is defined in degree $n$ by a product
\begin{equation}\label{equation-hom}
[\sheafhom_{qc}(X,Y)]_n = \prod^{qc}_{k \in \Z}
\sheafhom_{qc} (X_k,Y_{k+n})
\end{equation}
and again the differential involves a sign trick.
 Here we recall that products in $\qcox$ are also more subtle than one would like, so we are  using the notation  $\prod^{qc} F_i$ to denote the product of sheaves in $\qcox$.

\section{Model structures for the coderived and contraderived categories of a scheme}\label{sec-models}

This brief section is meant to familiarize the reader with some already known model structures for the coderived and contraderived categories. See \cite{positselski, becker, gillespie-mock projectives, gillespie-models-for-hocats-of-injectives, stovicek-purity, bravo-gillespie-hovey}. Throughout,  $\mathbb{X}$ denotes a semiseparated and quasi-compact scheme.  A key point is that, for the flat case, substituting $\dgclass{C} = \dwclass{C}$ (Theorem~\ref{them-cotorsion-complexes-scheme}) now gives us a much nicer description of the contraderived category. These model structures will be the backbone to our work in Sections~\ref{sec-Quillen-dualizing} and~\ref{sec-G-equivalence}. But we also use the model structure for the coderived category in Section~\ref{sec-recollement}.

\subsection{The injective model structure for the coderived category}
We let $\dwclass{I}$ denote the class of all  complexes of injective sheaves, that is, complexes that are \emph{degreewise} injective. We then let $\class{W}_{\textnormal{co}}$ denote the class of all complexes $W$ such that $W \xrightarrow{} I$ is null homotopic whenever $I \in \dwclass{I}$. Following Positselski, such complexes are called \textbf{coacyclic}. All contractible complexes are coacyclic, and all coacyclic complexes are exact. But even for modules over a  ring, an acyclic complex need not be coacyclic. It is known that $(\class{W}_{\textnormal{co}}, \dwclass{I})$ is an \emph{injective cotorsion pair}, meaning it determines an abelian model structure on $\chqcox$ in which the trivially fibrant objects coincide with the injective objects, equivalently, every object is cofibrant. The following proposition sums all this up.

\begin{proposition}\label{prop-coderived}
There is a cofibrantly generated hereditary abelian model structure $$\mathfrak{M}^{inj}_{\textnormal{co}} = (All, \class{W}_{\textnormal{co}}, \dwclass{I})$$ on $\chqcox$. Its homotopy category, $\textnormal{Ho}(\mathfrak{M}^{inj}_{\textnormal{co}})$, is called the \textbf{coderived category} and it is equivalent to $K(Inj)$, the chain homotopy category of all complexes of injective sheaves. That is, $\textnormal{Ho}(\mathfrak{M}^{inj}_{\textnormal{co}}) \cong K(Inj)$.
\end{proposition}

\begin{proof}
This goes back to~\cite[Cor.~4.4]{bravo-gillespie-hovey} where it was called the \emph{Inj model structure} and~\cite[Prop.~1.3.6]{becker} where it was called the \emph{coderived model structure}. The theorem holds over any Grothendieck category such as $\qcox$, by~\cite[Them.~4.2]{gillespie-models-for-hocats-of-injectives} or~\cite[Prop.~6.9]{stovicek-purity}.
\end{proof}

\subsection{The projective model structure for the contraderived category of an affine scheme}
Let $\mathbb{X} = Spec(R)$ of a commutative ring $R$. In this case, we also have the cotorsion pair $(\dwclass{P},\class{W}_{\textnormal{ctr}})$ in $\ch$. Here, $\dwclass{P}$ is the class of all degreewise projective complexes and the complexes in $\class{W}_{\textnormal{ctr}}$ have been called \textbf{contraacyclic}. The class  $\class{W}_{\textnormal{ctr}}$ is also thick, and this time the cotorsion pair $(\dwclass{P},\class{W}_{\textnormal{ctr}})$ is a \emph{projective cotorsion pair}, giving rise to a model structure in which all complexes are fibrant, as follows.

\begin{proposition}\label{prop-contraderived}
Let $R$ be any ring.
There is a cofibrantly generated hereditary abelian model structure $$\mathfrak{M}^{proj}_{\textnormal{ctr}} = (\dwclass{P}, \class{W}_{\textnormal{ctr}}, All)$$ on $\ch$. Its homotopy category, $\textnormal{Ho}(\mathfrak{M}_{\textnormal{ctr}})$, is called the \textbf{contraderived category} and it is equivalent to $K(Proj)$, the chain homotopy category of all complexes of projectives. That is, $\textnormal{Ho}(\mathfrak{M}^{proj}_{\textnormal{ctr}}) \cong K(Proj)$.
\end{proposition}

\begin{proof}
Again, this model was constructed in~\cite[Cor.~6.4]{bravo-gillespie-hovey} where it was called the \emph{Proj model structure} and in~\cite[Prop.~1.3.6]{becker} where it was called the \emph{contraderived model structure}.
\end{proof}

\subsection{The flat model structure for the contraderived category}

If we are in a Grothendieck category that is lacking enough projectives, then assuming the category has a tensor product, we try to use flat objects as replacements.  In particular, for a semiseparated Noetherian scheme $\mathbb{X}$, Murfet studied $$\class{D}(Flat) = K(Flat)/\tilclass{F},$$ the  \emph{derived category of flat quasi-coherent sheaves}. He shows this to be the correct replacement of $K(Proj)$ for non-affine schemes.  Since then however, we have the important realization Theorem~\ref{them-cotorsion-complexes-scheme}. It follows that $K(Flat\text{-}Cot)$, the homotopy category of all sheaves that are both flat and cotorsion, is equivalent to $\class{D}(Flat)$.
Indeed we can describe the equivalence $\class{D}(Flat) \cong K(Flat\text{-}Cot)$ from the abelian model category point of view as follows.
In~\cite[Section~4]{gillespie-mock projectives}, we can find the construction of a flat abelian  model structure for the contraderived category. It exists  on $\chqcox$ for any scheme $\mathbb{X}$ with a flat generator. Substituting $\dgclass{C} = \dwclass{C}$ gives us the following result. 

\begin{theorem}[Corollary 4.1 of \cite{gillespie-mock projectives} with Theorem 3.3 of~\cite{cet-G-flat-stable-scheme}]\label{them-contra-model-scheme}
There is a cofibrantly generated hereditary abelian model structure $$\mathfrak{M}^{flat}_{\textnormal{ctr}} = (\dwclass{F},  \class{W}_{\textnormal{ctr}},\dwclass{C})$$ on $\chqcox$.
Consequently, we have triangulated equivalences
$$\textnormal{Ho}(\mathfrak{M}^{flat}_{\textnormal{ctr}}) \cong K(Flat\text{-}Cot) \cong \class{D}(Flat).$$
\end{theorem}
It is the fact $\dwclass{F}\cap \class{W}_{\textnormal{ctr}} = \tilclass{F}$, the class of all exact complexes with flat cycles, that implies this is a model structure for Murfet's category $\class{D}(Flat)$.  In fact, by~\cite[Prop.~5.3(3)]{gillespie-exact model structures}, $\mathfrak{M}^{flat}_{\textnormal{ctr}}$ restricts to an  exact model structure $(\dwclass{F},  \tilclass{F},\dwclass{FC})$ on $\cha{F}$. Here, $\class{FC}$ denotes the class of all flat-cotorsion sheaves and $\cha{F}$ denotes the exact category of all chain complexes of flat sheaves. So this is an exact model structure on $\cha{F}$ with $\tilclass{F}$ as the class of trivial objects; clearly its homotopy category is the derived category of flats, as can be defined for general exact categories as in~\cite{neeman-exact category}.

We use the notation $\class{W}_{\textnormal{ctr}}$ because in the affine case this indeed is precisely the class of contraacyclic complexes appearing in Proposition~\ref{prop-contraderived}.
 The key to this is a nontrivial result of Neeman from~\cite{neeman-flat}. In the notation here, it is the statement $\dwclass{F} \cap \class{W}_{\textnormal{ctr}} = \tilclass{F}$ for any ring $R$. So in the affine case we have the following. Although again, the ring need not be commutative.

\begin{theorembis}{them-contra-model-scheme}[Corollary 4.1 of \cite{gillespie-mock projectives} with Theorem 5.3 of~\cite{bazzoni-cortes-estrada} and~\cite{neeman-flat}]\label{them-contra-model}
Let $R$ be any ring. Then there is a cofibrantly generated hereditary abelian model structure $$\mathfrak{M}^{flat}_{\textnormal{ctr}} = (\dwclass{F},  \class{W}_{\textnormal{ctr}},\dwclass{C})$$ on $\ch$ where $\class{W}_{\textnormal{ctr}}$ is precisely the class of contraacyclic complexes. Consequently, we have triangulated equivalences
$$K(Proj) \cong \textnormal{Ho}(\mathfrak{M}^{proj}_{\textnormal{ctr}}) = \textnormal{Ho}(\mathfrak{M}^{flat}_{\textnormal{ctr}}) \cong K(Flat\text{-}Cot) \cong \class{D}(Flat).$$
\end{theorembis}

\section{The stable derived category and recollement}\label{sec-recollement}

In Neeman's paper~\cite{neeman-homotopy category of injectives}, the question is raised as to when the recollement of Krause from~\cite{krause-stable derived cat of a Noetherian scheme} holds. In this section we show that the fact $\dgclass{C} = \dwclass{C}$ (Theorem~\ref{them-cotorsion-complexes-scheme}) implies that the recollement holds for any semiseparated and quasi-compact scheme $\mathbb{X}$, and avoids the need for generators of finite projective dimension, as in~\cite{gillespie-models-for-hocats-of-injectives}.

First we will construct an injective model structure for Krause's category $S(Inj)$, the stable derived category.
Critical to the proof is the fact that for any semiseparated and quasi-compact scheme $\mathbb{X}$, the category $\qcox$ admits a generating set $\{U_i\}$ with each $U_i$ a flat sheaf. A straightforward proof of this fact, suggested by Neeman, is given in~\cite[Appendix~A, Lemma~1]{efimov-positselski}. It is very notable too that \emph{any} quasi-compact and quasi-separated scheme possessing a generating set of flat sheaves is necessarily  semiseparated, by~\cite{slavick-stovicek-flat-generators}. So with our approach of using flat generators this is the best result we can expect.

\begin{theorem}\label{them-exact Inj model structure}
Assume $\mathbb{X}$ is a semiseparated and quasi-compact scheme. Let $\exclass{I}$ denote the class of all exact (acyclic) complexes of injective quasi-coherent sheaves. There is an injective abelian model structure $$\mathfrak{M}^{inj}_{\textnormal{st}} = (All, \class{W}_{\textnormal{st}}, \exclass{I})$$ on $\chqcox$.
Its homotopy category is equivalent to $S(Inj)$, the chain homotopy category of all exact complexes of injectives. That is, $\textnormal{Ho}(\mathfrak{M}^{inj}_{\textnormal{st}}) \cong S(Inj)$. This is called the \textbf{injective stable derived category of $\mathbb{X}$} and it is a well-generated triangulated category.
\end{theorem}

\begin{proof}
As noted above, $\qcox$ admits a generating set $\{U_i\}$ with each $U_i$ a flat sheaf. Let
\[
\class{T}_{st} = \{D^{n} (U_i/C) \,|\, n\in\Z, U_i \in \{U_i\}, \forall C \subseteq U_i\} \bigcup \{S^{n} (U_i) \,|\, n\in\Z, U_i \in \{U_i\}  \}.
\]

Let us first show $\rightperp{\class{T}_{st}} =\exclass{I}$. For all subobjects $C \subseteq U_i$, consider the short exact sequence
$$0 \xrightarrow{} D^n(C) \xrightarrow{} D^n(U_i) \xrightarrow{} D^n(U_i/C) \xrightarrow{} 0.$$
For a chain complex $Y$, $\Ext^{1}_{\chqcox}(D^{n} (U_i/C),Y) = 0$ implies an epimorphism
$$\Hom_{\chqcox}(D^{n}(U_i),Y) \xrightarrow{} \Hom_{\chqcox}(D^{n}(C),Y).$$
By a well known adjunction this is equivalent to an  epimorphism
$$\Hom_{\mathbb{X}}(U_i,Y_n) \xrightarrow{} \Hom_{\mathbb{X}}(C,Y_n).$$
This is equivalent to each $Y_n$ being an injective object by the generalization of Baer's criterion to Grothendieck categories~\cite[Prop.~2.9]{stenstrom}. Conversely, if each $Y_n$ is injective, then
$$\Ext^{1}_{\chqcox}(D^{n} (U_i/C),Y) \cong \Ext^{1}_{\mathbb{X}}(U_i/C,Y_n) = 0.$$
We conclude that $\rightperp{\class{T}_{st}}$ is the class of all complexes of injectives $Y$ such that
$$\Ext^{1}_{\chqcox}(S^{n}(U_i),Y) = 0.$$

Next, consider the short exact sequences
$$0 \xrightarrow{} S^{n}(U_i) \xrightarrow{} D^{n+1}(U_i) \xrightarrow{} S^{n+1}(U_i) \xrightarrow{} 0.$$
The condition $\Ext^{1}_{\chqcox}(S^{n+1}(U_i),Y) = 0$ implies we have an epimorphism
$$\Hom_{\chqcox}(D^{n+1}(U_i),Y) \xrightarrow{} \Hom_{\chqcox}(S^{n}(U_i),Y).$$
Since $\{U_i\}$ is a set of generators this implies that $Y$ must be an exact complex. (For example, see~\cite[Lemma~2.4]{gillespie-degreewise-model-strucs}.)
Now for any exact complex $Y$ we have an isomorphism
$$\Ext^{1}_{\chqcox}(S^{n}(U_i),Y) \cong \Ext^{1}_{\mathbb{X}}(U_i, Z_{n}Y).$$
We conclude that $\rightperp{\class{T}_{st}}$ is the class of all exact complexes of injectives $Y$ such that
$$\Ext^{1}_{\mathbb{X}}(U_i,Z_nY) = 0$$ for all $U_i$ and all $n$. But by Theorem~\ref{them-cotorsion-complexes-scheme}, (that is, \cite[Theorem~3.3]{cet-G-flat-stable-scheme}), each $Z_nY$ is cotorsion, and so this Ext group vanishes because the $U_i$ are flat.
This completes the proof that $\rightperp{\class{T}_{st}} =\exclass{I}$.

Any generating set $\class{T}$ in a Grothendieck category cogenerates a complete cotorsion pair
$(\leftperp{(\rightperp{\class{T}})},\rightperp{\class{T}})$, by~\cite[Corollary~2.15(2)]{saorin-stovicek}. Setting $\class{W}_{\textnormal{st}} = \leftperp{\exclass{I}}$, this proves $(\class{W}_{\textnormal{st}},\exclass{I})$ is a complete cotorsion pair.

We use~\cite[Proposition~3.3]{bravo-gillespie-hovey}  to show $(\class{W}_{\textnormal{st}},\exclass{I})$ determines an (injective) abelian model structure $\mathfrak{M}^{inj}_{\textnormal{st}} = (All, \class{W}_{\textnormal{st}}, \exclass{I})$. It requires that we show $\class{W}_{\textnormal{st}}$ is thick, and that $\class{W}_{\textnormal{st}}$ contains all injective chain complexes.
To see that $\class{W}_{\textnormal{st}}$ is thick, first note that, because $\exclass{I}$
consists of complexes of injectives and is closed under suspensions, Lemma~\ref{lemma-homcomplex-basic-lemma} implies that $W\in \class{W}_{\textnormal{st}}$
if and only if $\homcomplex_{\mathbb{X}}(W,E)$ is exact for all $E\in \exclass{I}$.  Now
suppose we have a short exact sequence
\[
0 \xrightarrow{} W \xrightarrow{} V \xrightarrow{} Z \xrightarrow{} 0,
\]
where two out of three of the complexes are in $\class{W}_{\textnormal{st}}$, and let
$E\in \exclass{I}$.  Since $E$ is a complex of injectives, the resulting
sequence
\[
0 \xrightarrow{} \homcomplex_{\mathbb{X}}(Z,E) \xrightarrow{} \homcomplex_{\mathbb{X}}(V,E) \xrightarrow{}
\homcomplex_{\mathbb{X}}(W,E)\xrightarrow{} 0
\]
is still a short exact sequence.  Since two out of three of these complexes are
exact, so is the third. Therefore $\class{W}_{\textnormal{st}}$ is a thick class.
If $W$ is contractible, then $\homcomplex_{\mathbb{X}}(W,E)$ is obviously
exact for any $E$, so $W\in \class{W}_{\textnormal{st}}$.
This completes the proof that $\mathfrak{M}^{inj}_{\textnormal{st}}$ is an (injective) abelian model structure. The homotopy category of any such model category is triangulated; for example, see~\cite[Corollary~1.1.15]{becker}.

The arguments given above, showing that the cotorsion pair $(\class{W}_{\textnormal{st}},\exclass{I})$ is cogenerated by a set, reveal it to have an explicit set of \emph{generating monorphisms} in the sense of~\cite{hovey}. (They are the monomorphisms in the short exact sequences.) Alternatively, it is shown in~\cite{saorin-stovicek} that a set of generating monomorphisms always exists whenever a cotorsion pair is cogenerated by a set containing a family of generators. It means that the model structure is cofibrantly generated~\cite[Section~6]{hovey}, and since we have a cofibrantly generated model structure on a locally presentable (pointed) category, a main result from~\cite{rosicky-brown representability combinatorial model srucs} tells us that it is well generated in the sense of~\cite{neeman-well generated}. Also see~\cite[Chapter~7]{hovey-model-categories} for related results.

It is left to prove the claim that $\textnormal{Ho}(\mathfrak{M}^{inj}_{\textnormal{st}}) \cong S(Inj)$, the chain homotopy category of all acyclic complexes of injectives.
From the fundamental theorem of model categories~\cite[Theorem~1.2.10]{hovey-model-categories} we know that the homotopy category of this model structure is equivalent to $\exclass{I}/\sim$ where $\sim$ denotes the formal homotopy relation. However, it follows from~\cite[Corollary~4.8]{gillespie-exact model structures} that formally $f \sim g$ if and only if $g-f$ factors through an injective object. Since the categorically injective complexes are contractible, and since contractible complexes are trivial, this implies that $f \sim g$ if and only if $f$ and $g$ are chain homotopic in the usual sense; see~\cite[Lemma~5.1]{gillespie-hereditary-abelian-models}. So the homotopy category of bifibrant objects is exactly $S(Inj)$.
\end{proof}

Becker showed in~\cite{becker} that Krause's recollement $S(Inj) \xrightarrow{} K(Inj) \xrightarrow{} \class{D}(R)$, from~\cite{krause-stable derived cat of a Noetherian scheme}, holds for any ring $R$, even without the noetherian hypothesis.  Krause's original work in~\cite{krause-stable derived cat of a Noetherian scheme} was in the setting of separated noetherian schemes and locally noetherian categories. In~\cite{neeman-homotopy category of injectives}, Neeman also studies the homotopy category of injectives and asks what is the proper generality for which the recollement will hold. He produced an example of a locally noetherian Grothendieck category $\cat{G}$ in which products of acyclic complexes of injectives need not be acyclic. It implies the failure of recollement, and by~\cite{krause-stable derived cat of a Noetherian scheme}, $\class{D}(\class{G})$ is not compactly generated.

But now we have the following consequence of Theorem~\ref{them-exact Inj model structure} (so again, the key is~\cite[Theorem~3.3]{cet-G-flat-stable-scheme}).

\begin{corollary}\label{cor-recollement of krause}
Assume $\mathbb{X}$ is a semiseparated and quasi-compact scheme. Then the canonical functors $S(Inj) \xrightarrow{} K(Inj) \xrightarrow{} \class{D}(\qcox)$ induce a recollement.
\end{corollary}

\begin{proof}
We use the generalization given in~\cite[Theorem~4.6]{gillespie-recollement} of Becker's method from~\cite{becker}. Let $(\class{E}, \dgclass{I})$ denote the cotorsion pair where $\class{E}$ is the class of exact complexes and $\dgclass{I}$ is the class of DG-injective complexes~\cite[Corollary~7.1]{gillespie-quasi-coherent}. Then the three cotorsion pairs $$\mathfrak{M}^{inj}_{\textnormal{co}} = (\class{W}_{\textnormal{co}},\dwclass{I}) \ , \ \ \ \ \mathfrak{M}^{inj}_{\textnormal{st}}  = (\class{W}_{\textnormal{st}},\exclass{I}) \ , \ \ \ \ \mathfrak{M}^{inj} = (\class{E}, \dgclass{I})$$
satisfy the hypotheses of~\cite[Theorem~4.6]{gillespie-recollement} and immediately give the recollement. Also, see~\cite[Theorem~3.5]{gillespie-mock projectives} for a description of how the functors work when the recollement is written in the form:
$$\textnormal{Ho}(\mathfrak{M}^{inj}_{\textnormal{st}}) \xrightarrow{} \textnormal{Ho}(\mathfrak{M}^{inj}_{\textnormal{co}}) \xrightarrow{} \textnormal{Ho}(\mathfrak{M}^{inj}).$$
\end{proof}

Even without the recollement, the inclusion $S(Inj) \xrightarrow{} K(Inj)$ has a right adjoint, which takes the form of a special precover (right approximation) by an exact complex. It follows formally that $S(Inj)$ has products, although the product may not be pretty after taking an exact precover. However, the recollement now provides a left adjoint to $S(Inj) \xrightarrow{} K(Inj)$. So being a right adjoint, the inclusion $S(Inj) \xrightarrow{} K(Inj)$ preserves products. It proves the following corollary.

\begin{corollary}\label{cor-products-acyclic}
Assume $\mathbb{X}$ is a semiseparated and quasi-compact scheme. Then the inclusion $S(Inj) \xrightarrow{} K(Inj)$ preserves products. Acyclic complexes of injective sheaves are closed under products.
\end{corollary}

We also have the following proof of the corollary.

\begin{proof}
Let $\{E_i\}$ be a collection of exact chain complexes of injective sheaves. We want to show that $\prod^{qc} E_i \in \rightperp{\class{T}_{st}}$. But given an object $T \in \class{T}_{st}$, there is a canonical monomorphism
$$\Ext^{1}_{\mathbb{X}}(T,\prod^{qc} E_i ) \xrightarrow{} \prod \Ext^{1}_{\mathbb{X}}(T, E_i ).$$
This canonical monomorphism always exists for abelian categories with products, see~\cite[Prop.~A.1]{stovicek-cotilting-Noetherian-schemes}.
Since the product on the right side is 0, we conclude $$\prod^{qc} E_i \in \exclass{I}.$$
\end{proof}

\section{Totally acyclic complexes of injectives and Gorenstein injectives}\label{sec-totally-acyclic-inj}

Assume $\mathbb{X}$ is a semiseparated Noetherian scheme.  The main point we wish to make here is that the Gorenstein injectives are the fibrant objects of a cofibrantly generated model structure on $\qcox$, and Quillen equivalent to one on $\chqcox$. Much of this essentially goes back to~\cite[Theorem~7.12]{krause-stable derived cat of a Noetherian scheme}. We wish to show here that the homotopy categories expand to include all objects of $\qcox$ and $\chqcox$, their models are cofibrantly generated, and lastly they are connected by a Quillen equivalence.

\subsection{A model for the homotopy category of totally acyclic complexes of injectives}
We wish to put a cofibrantly generated model structure on $\qcox$ whose homotopy category is $K_{to}(Inj)$, the homotopy category of \emph{totally acyclic complexes of injectives}. That is, exact complexes $X$ of injectives for which $\Hom_{\mathbb{X}}(I,X)$ remains exact for all injective sheaves $I$. Note that this is the categorical $\Hom_{\mathbb{X}}$ landing in abelian groups.

\begin{theorem}\label{them-totally Inj model structure}
Assume $\mathbb{X}$ is a semiseparated Noetherian scheme. Let $\toclass{I}$ denote the class of all totally acyclic complexes of injectives. There is an injective abelian model structure $$\mathfrak{M}^{inj}_{\textnormal{to}} = (All, \class{W}_{\textnormal{to}}, \toclass{I})$$ on $\chqcox$.
Its homotopy category is equivalent to $K_{to}(Inj)$, the chain homotopy category of all totally acyclic complexes of injectives. That is, $\textnormal{Ho}(\mathfrak{M}^{inj}_{\textnormal{to}}) \cong K_{to}(Inj)$ and it is a well-generated triangulated category.
\end{theorem}

\begin{proof}
Since $\mathbb{X}$ is semiseparated and quasi-compact, we may continue with the notation from the proof of Theorem~\ref{them-exact Inj model structure}. So  $\{U_i\}$ denotes a set of flat generators for $\qcox$, and we have
\[
\class{T}_{st} = \{D^{n} (U_i/C) \,|\, n\in\Z, U_i \in \{U_i\}, \forall C \subseteq U_i\} \bigcup \{S^{n} (U_i) \,|\, n\in\Z, U_i \in \{U_i\}  \}.
\]
Note that by~\cite[Prop.~II.7.17]{hartshorne-residues and duality}, see also~\cite[Appendix~B]{stovicek-cotilting-Noetherian-schemes}, there is a set $\{J(x)\}_{x\in \mathbb{X}}$ of indecomposable injective sheaves; every injective sheaf $I$ is a direct sum of sheaves $J(x)$ for various $x$ in the underlying space of $\mathbb{X}$.
So then a complex $X$ is a totally acyclic complex of injectives if and only if $X$ is an exact complex of injectives for which $\Hom_{\mathbb{X}}(J(x),X) = \homcomplex_{\mathbb{X}}(S^0J(x),X)$ is exact for all $x \in \mathbb{X}$. (Since this is the categorical $\Hom$, we are using that any direct product of exact complexes of \emph{abelian groups} is again exact.) We now set
\[
\class{T}_{to} = \class{T}_{st} \,\bigcup \,\{S^{n}(J(x)) \,|\, n\in\Z, x \in \mathbb{X}  \}.
\]
So $\rightperp{\class{T}_{to}}$ consists precisely of the exact complexes of injectives $X$ for which $$\Ext^1_{\chqcox}(S^nJ(x),X) = 0.$$ Since $X$ is a complex of injectives, this is equivalent to saying that $X$ is an exact complex of injectives for which $\Hom_{\mathbb{X}}(J(x), X)$ is exact. As noted, such $X$ are precisely the totally acyclic complexes of injectives. So $\class{T}_{to}$ cogenerates $(\class{W}_{to},\toclass{I})$.
The remaining statements are proved in the exact same way that we proved Theorem~\ref{them-exact Inj model structure}. Simply replace   $\class{W}_{st}$, $\exclass{I}$, and $S(Inj)$ in that proof with $\class{W}_{to}$, $\toclass{I}$, and $K_{to}(Inj)$.
\end{proof}

\subsection{Gorenstein injectives}
The model structure of Theorem~\ref{them-totally Inj model structure} passes to a Quillen equivalent one on $\qcox$.
We say that a sheaf $M$ is \textbf{Gorenstein injective} if
$M=Z_{0}I$ for some totally acyclic complex $I$ of injectives. We let $\class{GI}$ denote the class of all Gorenstein injective sheaves  and set $\class{W}=\leftperp{\class{GI}}$.

\begin{theorem}\label{them-G-inj}
Assume $\mathbb{X}$ is a semiseparated Noetherian scheme. Then the triple
$$\mathfrak{M}^{inj}_{\textnormal{G}} = (All, \class{W}, \class{GI})$$ is a cofibrantly generated abelian model structure on $\qcox$. We call it the \textbf{Gorenstein injective model structure}. Moreover, the functor
$$S^0(-) : \mathfrak{M}^{inj}_{\textnormal{G}}  \xrightarrow{} \mathfrak{M}^{inj}_{\textnormal{to}}$$ is a (left) Quillen equivalence with the cycle functor $Z_0$ as its (right adjoint) Quillen inverse.
\end{theorem}

\begin{proof}
The root of this idea goes back to~\cite[Theorem~7.12]{krause-stable derived cat of a Noetherian scheme}. However, it takes some more work to show that we get a complete hereditary cotorsion pair, cogenerated by a set. It is shown in~\cite[Lemmas~3.4/3.5]{estrada-iacob-frontiers-china}   that the Gorenstein injective sheaves form the right side of a complete hereditary cotorsion pair. It is then automatic that we get the model structure. We would still like to see that it is cofibrantly generated. For this, we note that we can imitate the method of~\cite[Section~7]{gillespie-models-for-hocats-of-injectives}. The techniques there are based on~\cite{bravo-gillespie-hovey}. However there is one point of confusion --- Gillespie proved the analog of the current theorem but by assuming the category had a set of noetherian generators of finite projective dimension. In hindsight, the notherian generators may be distinct from the generators $\{G_i\}$ of finite projective dimension in~\cite[Section~7 Setup]{gillespie-models-for-hocats-of-injectives}. Moreover, since we now know $\dgclass{C} = \dwclass{C}$, we may discard this set of generators completely and replace it with the set $\{U_i\}$ of flat generators. Indeed $\{U_i\} \subseteq \class{W}$ which now replaces~\cite[Lemma~7.2]{gillespie-models-for-hocats-of-injectives}. With these adjustments made, everything in~\cite[Section~7]{gillespie-models-for-hocats-of-injectives} will now carry over to prove the current theorem. The Quillen equivalence argument from~\cite[Theorem~5.8]{bravo-gillespie-hovey} also generalizes.
\end{proof}

\section{F-totally acyclic complexes  and Gorenstein flats}\label{sec-F-totally}

We are still assuming $\mathbb{X}$ be any semiseparated Noetherian scheme. In this section we show that the Gorenstein flat model structure on $\qcox$, from~\cite{cet-G-flat-stable-scheme}, is Quillen equivalent to the F-totally acyclic model structure on $\chqcox$, from~\cite{estrada-gillespie-coherent-schemes}. We previously only had a proof of this in the affine case, for coherent rings. But that proof relied on the existence of an equivalent projective model structure (the Ding projectives) on $R$-Mod; see~\cite[Theorem~5.1]{estrada-gillespie-coherent-schemes}. A proof avoiding projectives eluded us while writing~\cite{estrada-gillespie-coherent-schemes}. However, we now know $\dgclass{C} = \dwclass{C}$, and this is a key ingredient to the proof below.

First, we recall the two model structures.

\begin{theorem}\cite[Theorem~2.5]{cet-G-flat-stable-scheme}\label{them-G-flat}
Assume $\mathbb{X}$ is a semiseparated Noetherian scheme. Let $\class{GF}$ denote the class of Gorenstein flat sheaves and $\class{C}$ the class of cotorsion sheaves. Then there  is a cofibrantly generated abelian model structure
$$\mathfrak{M}^{flat}_{\textnormal{G}} = (\class{GF}, \class{V}, \class{C})$$ on $\qcox$. We call it the \textbf{Gorenstein flat model structure}.
We have an equivalence of triangulated categories $$\textnormal{Ho}(\mathfrak{M}^{flat}_{\textnormal{G}}) \simeq \textnormal{St}(G\textnormal{-}Flat\textnormal{-}Cot).$$
\end{theorem}

\begin{theorem}\cite[Theorem~1.2]{estrada-gillespie-coherent-schemes}\label{them-F-flat}
Assume $\mathbb{X}$ is a semiseparated Noetherian scheme. Let ${}_I\tilclass{F}$ denote the class of F-totally acyclic complexes of flat sheaves and $\dwclass{C}$ the class of all complexes of cotorsion sheaves. Then there  is a cofibrantly generated abelian model structure
$$\mathfrak{M}^{flat}_{\textnormal{F-to}} = ({}_I\tilclass{F}, \class{V}_{\textnormal{to}}, \dwclass{C})$$ on $\chqcox$. We call it the \textbf{F-totally acyclic flat model structure}. We have an equivalence of triangulated categories $$\textnormal{Ho}(\mathfrak{M}^{flat}_{\textnormal{F-to}}) \simeq K_{F\textnormal{-}to}(Flat\textnormal{-}Cot).$$
\end{theorem}

After some needed lemmas, we will prove that the two model structures above are Quillen equivalent.

\begin{lemma}\label{lemma-transfinite extensions of spheres and disks}
Let $\cat{A}$ be a bicomplete abelian category and let $(\class{X},\class{Y})$ be a cotorsion pair of chain complexes in $\cha{A}$. Suppose $\class{C}$ is some given class of objects in $\cat{A}$.
\begin{enumerate}
\item If the spheres $S^n(C)$ are in $\class{X}$ whenever $C$ is in $\class{C}$, then any bounded below complex with entries in $\class{C}$ is also in $\class{X}$.
\item If the disks $D^n(C)$ are in $\class{X}$ whenever $C$ is in $\class{C}$, then any bounded above exact complex with cycles in $\class{C}$ is also in $\class{X}$.
\item If the spheres $S^n(C)$ are in $\class{Y}$ whenever $C$ is in $\class{C}$, then any bounded above complex with entries in $\class{C}$ is also in $\class{Y}$.
\item If the disks $D^n(C)$ are in $\class{Y}$ whenever $C$ is in $\class{C}$, then any bounded below exact complex with cycles in $\class{C}$ is also in $\class{Y}$.
\end{enumerate}
\end{lemma}

\begin{proof}
Note that (1) and (3) are dual statements and (2) and (4) are dual. We will prove (1) and (4). For (1), suppose that $(X,d)$ is a bounded below complex with entries in $\class{C}$. It is easy to check that $X$ can be expressed as a transfinite extension of spheres $S^n(X_n)$, on the components $X_n$. Each $S^n(X_n)$ is in $\class{X}$ by hypothesis and so $X$ is in $\class{X}$ too by the Eklof Lemma.

Next we prove (4). Here we note that any bounded below exact complex $(X,d)$ can be expressed as an inverse transfinite extension as indicated in the diagram:
$$\begin{CD}
             @.             @.      0  @<<< Z_3X  @<<< \cdots \\
  @.         @.                    @VVV         @VVV      \\
             @.        0     @<<<    Z_2X    @<<d< X_3  @= \cdots \\
  @.         @VVV                    @VVV         @VdVV      \\
  0           @<<<  Z_1X           @<<d<    X_2    @= X_2  @= \cdots \\
  @VVV         @VVV                    @VdVV         @VdVV      \\
  X_0 @<<d< X_1            @=  X_1           @=    X_1    @= \cdots   \\
  @|         @VdVV                    @VdVV         @VdVV      \\
  X_0 @= X_0            @=  X_0           @=    X_0    @= \cdots   \\
  @VVV         @VVV                    @VVV         @VVV      \\
  0   @.       0              @.        0              @.        0        @. \\
\end{CD}$$ Indeed note that each horizontal map in the diagram is surjective with its kernel being a disk $D^{n+1}(Z_nX)$. So $X$ is an inverse transfinite extension of the disks $D^{n+1}(Z_nX)$. The desired result now follows from a generalized version of Trlifaj's~\cite[Lemma~2.3]{trlifaj-Ext and inverse limits}, which is the dual of Eklof's Lemma. A nice treatment of the needed fact for a complete abelian category can be found in~\cite[Lemma~6.8]{holm-jorgensen-quiver-reps}.
\end{proof}

\begin{lemma}\label{lemma-bounded complexes}
The following complexes are in $\rightperp{{}_I\tilclass{F}}$.
\begin{enumerate}
\item Any bounded below exact complex with cotorsion cycles is in $\rightperp{{}_I\tilclass{F}}$.
\item Any bounded above complex of Gorenstein cotorsion sheaves is in $\rightperp{{}_I\tilclass{F}}$.
\end{enumerate}
\end{lemma}

\begin{proof}
Let $F \in {}_I\tilclass{F}$. Each $F_n$ is flat, so for any cotorsion $C$ we have $$0 = \Ext^1_{\mathbb{X}}(F_n,C) \cong \Ext^1_{\chqcox}(F,D^{n+1}(C)).$$ Hence each $D^n(C) \in \rightperp{{}_I\tilclass{F}}$, and (1) follows from Lemma~\ref{lemma-transfinite extensions of spheres and disks}(4). Secondly, each $F_n/B_nF$ is Gorenstein flat. So for any Gorenstein cotorsion $C$ we have $$0 = \Ext^1_{\mathbb{X}}(F_n/B_nF,C) \cong \Ext^1_{\chqcox}(F,S^n(C)).$$ (A proof of the isomorphism can be found in~\cite[Lemma~4.2]{gillespie-degreewise-model-strucs}.) Hence each $S^n(C) \in \rightperp{{}_I\tilclass{F}}$, and (2) follows from Lemma~\ref{lemma-transfinite extensions of spheres and disks}(3).
\end{proof}

\begin{lemma}\label{lemma-sphere}
Let $L$ be a sheaf. Then $S^0(L) \in \rightperp{{}_I\tilclass{F}}$ if and only if $L$ is Gorenstein cotorsion.
\end{lemma}

\begin{proof}
Suppose $S^0(L) \in \rightperp{{}_I\tilclass{F}}$ and let $M = F_0/B_0F$ be a Gorenstein flat sheaf, arising from the F-totally acyclic complex of flats, $F$. Then the isomorphism $$\Ext^1_{\mathbb{X}}(F_0/B_0F,L) \cong \Ext^1_{\chqcox}(F,S^0(L)) = 0$$ tells us $L$ is Gorenstein cotorsion. On the other hand, if $L$ is Gorenstein cotorsion and $F$ is any F-totally acyclic complex of flats, then the same isomorphism tells us $S^0(L) \in  \rightperp{{}_I\tilclass{F}}$.
\end{proof}

\begin{theorem}\label{them-G-Quillen-equiv}
Assume $\mathbb{X}$ is a semiseparated Noetherian scheme.
Then the functor
$$S^0(-) : \mathfrak{M}^{flat}_{\textnormal{G}}  \xrightarrow{} \mathfrak{M}^{flat}_{\textnormal{F-to}}$$ is a right Quillen equivalence. The functor $G : \chqcox \xrightarrow{} \qcox$ defined by the rule $G(X) = X_0/B_0X$ is the left adjoint, so this is the (left) Quillen inverse.
\end{theorem}

\begin{proof}
Note that for any scheme $\mathbb{X}$, the functor $G$ is indeed a left adjoint. Its right adjoint is the functor $S^0$ which turns a sheaf into a complex concentrated in degree zero. To see that $G$ is left Quillen, let $X \xrightarrow{f} Y$ be a cofibration in the F-totally acyclic flat model structure. Then by definition we have a short exact sequence $0 \xrightarrow{} X  \xrightarrow{f} Y \xrightarrow{} F \xrightarrow{} 0$ with $F$ an F-totally acyclic complex of flat sheaves. The functor $G$ is only right exact in general, but since $F$ is exact we do get a short exact sequence $0 \xrightarrow{} X_0/B_0X  \xrightarrow{G(f)} Y_0/B_0Y \xrightarrow{} F_0/B_0F \xrightarrow{} 0$. We have $F_0/B_0F \cong Z_{-1}F$ which is Gorenstein flat. Since $G(f)$ is a monomorphism with Gorenstein flat cokernel, $G$ preserves cofibrations. The same argument shows that $G$ preserves trivial cofibrations; we just replace $F$ with an acyclic complex with all cycles flat.

To show that $G$ is a Quillen equivalence we use~\cite[Corollary~1.3.16(b)]{hovey-model-categories}. In our case it means we must prove the following: (i)  If $X \xrightarrow{f} Y$ is a chain map between two F-totally acyclic complexes of flats for which $X_0/B_0X \xrightarrow{G(f)} Y_0/B_0Y$ is a weak equivalence, then $f$ itself must be a weak equivalence. (ii) For all cotorsion sheaves $C$, the map $GQS^0(C) \xrightarrow{} C$, where $Q$ is cofibrant replacement, is a weak equivalence.

To prove (i), we use the factorization axiom to write $f = pi$ where $X \xrightarrow{i} Z$ is a trivial cofibration and $Z \xrightarrow{p} Y$ is a fibration. We note that we have short exact sequences $0 \xrightarrow{} X  \xrightarrow{i} Z \xrightarrow{} F \xrightarrow{} 0$ and  $0 \xrightarrow{} K  \xrightarrow{} Z \xrightarrow{p} Y \xrightarrow{} 0$. Since $F \in \tilclass{F}$,  it is F-totally acyclic and hence $Z$ and therefore $K$ must be too. Applying $G$ to this factorization gives us short exact sequences
$$0 \xrightarrow{} X_0/B_0X  \xrightarrow{G(i)} Z_0/B_0Z \xrightarrow{} F_0/B_0F \xrightarrow{} 0$$ and  $$0 \xrightarrow{} K_0/B_0K  \xrightarrow{} Z_0/B_0Z \xrightarrow{G(p)} Y_0/B_0Y \xrightarrow{} 0$$ and a factorization $G(f) = G(p)G(i)$. As already shown above, $G(i)$ is a trivial cofibration. Also $G(f)$ is a weak equivalence by hypothesis. So by the two out of three axiom, $G(p)$ must also be a weak equivalence. Being a surjection, it means that $\ker{(G(p))} = K_0/B_0K$ must be trivial. But we also have that $K$ is F-totally acyclic, thus $\ker{(G(p))}$ is Gorenstein flat.  This means $\ker{(G(p))}$ is trivially cofibrant; that is, a flat sheaf. But the class of trivial sheaves is thick and contains the flat sheaves, and hence each $Z_nK$  must be trivially cofibrant (flat). This proves that $\ker{p} \in \tilclass{F}$, and so $p$ is a trivial fibration. Therefore $f$ is a weak equivalence, and we have proved (i).

To prove (ii), let $C$ be a cotorsion sheaf. To get a cofibrant replacement of $S^0C$ in the F-totally acyclic flat model structure,  we use enough projectives of the cotorsion pair
$({}_I\tilclass{F}, \rightperp{{}_I\tilclass{F}})$. This gives us a short exact sequence
\begin{equation}\label{equation-comp-precover}\tag{$*$}
0 \xrightarrow{} Y \xrightarrow{} F \xrightarrow{} S^0C \xrightarrow{} 0
\end{equation}
with $F$ an F-totally acyclic complex of flats and $Y \in \rightperp{{}_I\tilclass{F}}$.
Applying $G$ to the short exact sequence gives us, using the snake lemma, another short exact sequence
\begin{equation}\label{equation-mods-precover}\tag{$**$}
0 \xrightarrow{} Y_0/B_0Y \xrightarrow{} F_0/B_0F \xrightarrow{} C \xrightarrow{} 0.
\end{equation}
The problem is to show that $F_0/B_0F \xrightarrow{} C$ is a weak equivalence. Being an epimorphism, it is enough to show its kernel is a Gorenstein cotorsion sheaf. So our goal is to show that $\Ext^1_{\mathbb{X}}(M,C) = 0$ for any Gorenstein flat sheaf $M$.
Below we have written some of the short exact sequence in~\eqref{equation-comp-precover}  in more detail.
\[
  \begin{CD}
      @VVV                               @VVV                                  @VVV       \\
     Y_{2} @=  F_{2}    @>>> 0 \\
      @Vd_2VV                              @Vd_2VV                                  @VVV       \\
    Y_{1} @=    F_{1}                 @>>>              0           \\
      @Vd_1VV                               @Vd_1VV                                  @VVV       \\
     Y_0         @>>>        F_0       @>>>   C \\
      @Vd_0VV                              @Vd'_0VV                                  @VVV       \\
        Y_{-1}         @=   F_{-1}          @>>>  0 \\
      @VVV                              @VVV                                  @VVV       \\
  \end{CD}
\]
Now each component $Y_n$ must be cotorsion, and each component of $F_n$ is flat.
Since we are given that $C$ is a cotorsion sheaf, the components of $F$ are all flat-cotorsion sheaves. In particular, $F$ is an acyclic complex of cotorsion sheaves and so each of its cycles must also be cotorsion by Theorem~\ref{them-cotorsion-complexes-scheme} (that is, \cite[Theorem~3.3]{cet-G-flat-stable-scheme}). Plainly, $B_0Y = B_0F$ is one of these cotorsion cycles. So let us note right now that the complex below is an exact complex with cotorsion cycles
$$\cdots \xrightarrow{} Y_2 \xrightarrow{d_2} Y_1 \xrightarrow{d_1} B_0Y \xrightarrow{} 0. $$
The complex $Y$ in~\eqref{equation-comp-precover} is an extension of the complex just written above, as indicated in the diagram below.
\[
  \begin{CD}
      @VVV                               @VVV                                  @VVV       \\
     Y_{2} @=  Y_{2}    @>>> 0 \\
      @Vd_2VV                              @Vd_2VV                                  @VVV       \\
    Y_{1} @=    Y_{1}                 @>>>              0           \\
      @Vd_1VV                               @Vd_1VV                                  @VVV       \\
     B_0Y         @>>>     Y_0          @>>\pi_0>   Y_0/B_0Y \\
      @VVV                              @Vd_0VV                                  @V\bar{d}_0VV       \\
          0         @>>>   Y_{n-1}          @=  Y_{n-1} \\
      @VVV                              @VVV                                  @VVV       \\
  \end{CD}
\]
Now the middle complex $Y$ is in $\rightperp{{}_I\tilclass{F}}$, and the complex on the left is also in $\rightperp{{}_I\tilclass{F}}$, by Lemma~\ref{lemma-bounded complexes}. Since $\rightperp{{}_I\tilclass{F}}$ is a coresolving class (closed under cokernels of monomorphisms), the complex on the right must also be in $\rightperp{{}_I\tilclass{F}}$ :
$$0 \xrightarrow{} Y_0/B_0Y \xrightarrow{\bar{d}_0} Y_{-1} \xrightarrow{d_{-1}} Y_{-2} \xrightarrow{} \cdots $$
But in turn, this complex is a trivial extension as shown below:
\[
  \begin{CD}
      @VVV                               @VVV                                  @VVV       \\
     0      @= 0 @= 0\\
      @VVV                              @VVV                                  @VVV       \\
    0 @>>>    Y_0/B_0Y                 @=             Y_0/B_0Y          \\
      @VVV                                @V\bar{d}_0VV                                 @VVV       \\
     Y_{-1}        @= Y_{-1}          @>>>   0\\
      @Vd_{-1}VV                               @Vd_{-1}VV                                  @VVV       \\
         Y_{-2}          @=   Y_{-2}          @>>>  0 \\
      @VVV                              @VVV                                  @VVV       \\
  \end{CD}
\]
Now here, we have already deduced that the middle complex is in $\rightperp{{}_I\tilclass{F}}$. The complex on the left is a bounded above complex of cotorsion-flat sheaves. By \cite[Lemma~2.3]{cet-G-flat-stable-scheme} we know that flat-cotorsion sheaves are also Gorenstein cotorsion. So the complex on the left is also in $\rightperp{{}_I\tilclass{F}}$, by Lemma~\ref{lemma-bounded complexes}. Again, $\rightperp{{}_I\tilclass{F}}$ is a coresolving class, and so the complex on the right must also be in $\rightperp{{}_I\tilclass{F}}$. Thus we have now shown that $S^0(Y_0/B_0Y) \in\rightperp{{}_I\tilclass{F}}$. By Lemma~\ref{lemma-sphere}  we conclude that $Y_0/B_0Y$ is Gorenstein cotorsion. This proves (ii) and completes the proof that $(G, S^0)$ is a Quillen adjunction from  $\mathfrak{M}^{flat}_{\textnormal{F-to}}$ to $\mathfrak{M}^{flat}_{\textnormal{G}}$.
\end{proof}

\section{Dualizing complexes induce Quillen Equivalences}\label{sec-Quillen-dualizing}

So far we have at most qualified $\mathbb{X}$ to be a semiseparated Noetherian scheme. We now assume that it also admits a dualizing complex $D$. We note that, as in~\cite{iyengar-krause} and~\cite{murfet-salarian}, we are taking this to include the assumption that $D$ be a bounded complex of injectives.

Consider the affine case of a commutative Noetherian ring $R$.
Iyengar and Krause proved in~\cite{iyengar-krause} that for a commutative Noetherian ring $R$ with dualizing complex $D$, the functor $$D \tensor_R - : K(Proj) \xrightarrow{} K(Inj)$$ is an equivalence of triangulated categories. Murfet extended this to sheaves on a semiseparated  Noetherian scheme $\mathbb{X}$ in his thesis~\cite{murfet-thesis}, by replacing $K(Proj)$ with $\class{D}(Flat) = K(Flat)/\tilclass{F}$, where $\tilclass{F}$ is the class of pure acyclic complexes of flats. See also~\cite{murfet-salarian}.

We would like to give a model category interpretation of this. We already described in Section~\ref{sec-models} the abelian model category structures recovering $K(Inj)$ and $\class{D}(Flat)$, as well as $K(Proj)$ in the affine case.  So it is natural to wish to see how a dualizing complex $D$ yields a Quillen equivalence whose derived adjunction is an equivalence from the contraderived category to the coderived category. This is what we show in this section.

Our first step is to prove Proposition~\ref{prop-quillenfunctor} which says that tensoring with $D$ is a Quillen adjunction. The proof boils down to Iyengar and Krause's observation that, since $D$ is a complex of injectives and $R$ is Noetherian, $D \tensor - : \ch \xrightarrow{} \ch$ takes complexes of flat (and in particular projective) $R$-modules  to complexes of injective $R$-modules. Building on Neeman's~\cite[Section~9]{neeman-flat}, Murfet extends this to semiseparated Noetherian schemes in~\cite[Section~8]{murfet-thesis}. The following lemma records some of their fundamental results which will ultimately lead us to the Quillen equivalence.

\begin{lemma}\label{lemma-tensor-injectives}
Let $R$ be a commutative Noetherian ring and $I$ an injective $R$-module.
\begin{enumerate}
\item $I \tensor P$ is an injective $R$-module whenever $P$ is a projective $R$-module.
\item $I \tensor F$ is an injective $R$-module whenever $F$ is a flat $R$-module.
\item $\Hom_R(I,E)$ is a flat-cotorsion $R$-module whenever $E$ is injective.
\end{enumerate}
In general, let $\mathbb{X}$ be a semiseparated Noetherian scheme and $I$ an injective sheaf on $\mathbb{X}$.
Then $I \tensor F$ is an injective sheaf whenever $F$ is a flat sheaf.
On the other hand, $\sheafhom_{qc}(I,E)$ is a flat-cotorsion sheaf whenever $E$ is an injective sheaf.

Moreover, not only are injectives in $\qcox$ closed under direct sums, but flat-cotorsion sheaves are closed under direct products in $\qcox$.
\end{lemma}

\begin{proof}
Let us first consider the case of a commutative Noetherian ring $R$. Since direct sums of injective modules are again injective in this case, $I \tensor (\oplus R)$ is injective for any free module $\oplus R$. It follows that $I \tensor P$ is injective for any projective $R$-module $P$, since $P$ is a summand of a free $R$-module.
In fact, direct limits of injective $R$-modules are again injective whenever $R$ is Noetherian ring. So $I \tensor F$ is injective for any flat $F$, by Lazard's Theorem.

Now consider the case of a semiseparated Noetherian scheme $\mathbb{X}$, an injective sheaf $I$, and a flat sheaf $F$. By~\cite[Prop.~II.7.17]{hartshorne-residues and duality}, injectivity of a sheaf can be checked on stalks, and of course this is always the case for flatness. In particular, the tensor product $I \tensor F$ is injective if and only if each stalk $(I \tensor F)_x \cong I_x \tensor_{\class{O}_x} F_x$ is an injective ${\class{O}_x}$-module. The local ring ${\class{O}_x}$ is Noetherian for  any locally Noetherian scheme $\mathbb{X}$, so we have reduced to the above case of a Noetherian ring $R$.

Note that for modules over a commutative ring, $\Hom_R(I,E)$ isn't just flat whenever $I$ and $E$ are injective.  We also have that for any module $M$, $\Hom_R(M,E)$ is always pure-injective when $E$ is injective; for example, see the proof of~\cite[Prop.5.3.7]{enochs-jenda-book}. Pure-injective modules are clearly cotorsion modules, and cotorsion modules are always closed under products. Flat modules too are closed under products assuming the ring is Noetherian (or just coherent). For sheaves on $\mathbb{X}$, it is more difficult to make the analogous arguments for $\sheafhom_{qc}(I,E)$. Nevertheless, Murfet and Salarian have shown exactly this in~\cite[Lemmas~3.2/3.9]{murfet-salarian}.
\end{proof}

\begin{proposition}\label{prop-quillenfunctor}
Let $\mathbb{X}$ be a semiseparated Noetherian scheme admitting a dualizing complex $D$. Then the functor $D \tensor - : \mathfrak{M}^{flat}_{\textnormal{ctr}}  \xrightarrow{} \mathfrak{M}^{inj}_{\textnormal{co}}$ is a left Quillen functor. In fact, we have
\begin{enumerate}
\item $D \tensor F$ is a complex of injectives whenever $F$ is a complex of flats.
\item $D \tensor F$ is a contractible complex of injectives whenever $F$ is a pure acyclic complex of flats.~\textnormal{cf.}~\cite[Lemma~8.2]{murfet-thesis}.
\end{enumerate}

In the affine case of a commutative Noetherian ring $R$ admitting a dualizing complex $D$, then $D \tensor - : \mathfrak{M}^{proj}_{\textnormal{ctr}}  \xrightarrow{} \mathfrak{M}^{inj}_{\textnormal{co}}$ is also a left Quillen functor.
\end{proposition}

\begin{proof}
We are to show that $D \tensor -$ preserves cofibrations and trivial cofibrations. Cofibrations (resp. trivial cofibrations) are monomorphisms with cofibrant (resp. trivially cofibrant) cokernel. So for $i$ to be a cofibration (resp. trivial cofibration) in the flat model structure means we have a short exact sequence
$$0 \xrightarrow{} X \xrightarrow{i} Y \xrightarrow{} F \xrightarrow{} 0$$
where $F$ is a complex of flat sheaves (resp. $F$ an exact complex of sheaves with each $Z_nF$ flat). In particular, it is a pure exact sequence in each degree, and so for all pairs of integers $i, j$
we have a short exact sequence
$$0 \xrightarrow{} D_i \tensor X_j \xrightarrow{}D_i \tensor Y_j \xrightarrow{} D_i \tensor F_j \xrightarrow{} 0.$$
Since short exact sequences are closed under direct sums, it follows from the definition of tensor product, (see Equation~\eqref{equation-tensor} from Section~\ref{section-csms}), that we get a short exact sequence of chain complexes
$$0 \xrightarrow{} D \tensor X \xrightarrow{D \tensor i} D \tensor Y \xrightarrow{} D \tensor F \xrightarrow{} 0.$$
Since $D \tensor i$ is a monomorphism, it is at the very least a cofibration in the injective model structure, because \emph{everything} is cofibrant. But in the case that $F$ is an exact complex of sheaves with each $Z_nF$ flat, we need to show $D \tensor F \in \class{W}_{\textnormal{co}}$. That is, we need to show it is a coacyclic complex.

But as Murfet showed in his thesis, $D \tensor F$ is contractible, and any contractible complex is coacyclic.
It is easy enough, and enlightening, to see why $D \tensor F$ is contractible. So let us give an argument for the reader's convenience. The point is that for \emph{any} injective sheaf $I$, the complex $S^n(I) \tensor F = \Sigma^n I \tensor F$ is an exact complex and all its cycles are injective, by Lemma~\ref{lemma-tensor-injectives}. As such, $I \tensor F$ is a contractible complex with injective components. But the complex $D$ is a bounded complex of injectives, and so it is inductively build up as a finite extension of such spheres $S^n(I)$. In particular, if $k$ is the largest integer such that $D_k$ is  non-zero, then we have an obvious short exact sequence
$$0 \xrightarrow{} {}_{k-1}D  \xrightarrow{} D \xrightarrow{} S^k(D_k) \xrightarrow{} 0$$
corresponding to the inclusion ${}_{k-1}D \subseteq D$, where ${}_{k-1}D$ is the subcomplex identical to $D$ but with a 0 in degree $k$.
The simple sequence contains an isomorphism in each degree, so certainly (again by equation~\eqref{equation-tensor} in the definition of tensor product) we get another short exact sequence of chain complexes
$$0 \xrightarrow{} {}_{k-1}D \tensor F \xrightarrow{} D \tensor F \xrightarrow{} S^k(D_k) \tensor F \xrightarrow{} 0.$$
As we have observed, $S^k(D_k) \tensor F$ is an exact complex with injective cycles and by an inductive hypothesis ${}_{k-1}D \tensor F$ is too. Therefore the extension $D \tensor F $ is an exact complex with injective cycles. Therefore $D \tensor F$ is a contractible complex of injectives.
This completes the proof that $D \tensor - $ is indeed a Quillen functor. We have also proved statement (2). For statement (1), we note that direct sums of injectives are again injective in noetherian categories. So  Lemma~\ref{lemma-tensor-injectives}, along with the definition of tensor product, (again see Equation~\eqref{equation-tensor} from Section~\ref{section-csms}), proves statement (1).

For the affine case we can also consider the domain category along with the projective model structure. The proof is similar, but easier.
\end{proof}

It follows formally that the right adjoint $\sheafhom_{qc}(D,-)$ is a right Quillen functor.  That is, it preserves fibrations and trivial fibrations. But again it is even better than this. We have the following.

\begin{proposition}\label{prop-rightquillenfunctor}
Let $\mathbb{X}$ be a semiseparated Noetherian scheme admitting a dualizing complex $D$.
Then the right adjoint $\sheafhom_{qc}(D,-) : \mathfrak{M}^{inj}_{\textnormal{co}} \xrightarrow{} \mathfrak{M}^{flat}_{\textnormal{ctr}}$ is a right Quillen functor.  In fact,
\begin{enumerate}
\item $\sheafhom_{qc}(D,E)$ is a complex of flat-cotorsion sheaves whenever $E$ is a complex of injectives.
\item $\sheafhom_{qc}(D,E)$ is a contractible complex of flat-cotorsion sheaves whenever $E$ is a contractible complex of injective sheaves.~\textnormal{cf.}~\cite[Lemma~8.8]{murfet-thesis}.
\end{enumerate}
In the affine case of a commutative Noetherian ring $R$ admitting a dualizing complex $D$, then $\homcomplex_R(D,-) : \mathfrak{M}^{inj}_{\textnormal{co}}  \xrightarrow{} \mathfrak{M}^{proj}_{\textnormal{ctr}}$ is also a right Quillen functor.
\end{proposition}

\begin{proof}
Since $D \tensor -$ is left Quillen, it is automatic that $\sheafhom_{qc}(D,-)$ is right Quillen. In particular, $\sheafhom_{qc}(D,E)$ must be a complex of cotorsion sheaves whenever $E$ is a complex of injective sheaves. But Lemma~\ref{lemma-tensor-injectives}, along with Equation~\eqref{equation-hom} from Section~\ref{section-csms} where we gave the definition of $\sheafhom_{qc}$, tells us that $\sheafhom_{qc}(D,E)$ is a complex of flat-cotorsion sheaves whenever $E$ is a complex of injective sheaves. In other words, $\sheafhom_{qc}(D,E) \in \dwclass{F}\cap\dwclass{C}$. Moreover, should $E$ be contractible (so trivially fibrant), then it follows that
$$\sheafhom_{qc}(D,E) \in \dwclass{F}\cap\class{W}_{ctr}\cap\dwclass{C}.$$ But this intersection is precisely the class of contractible complexes with flat-cotorsion components.
\end{proof}

Note that Propositions~\ref{prop-quillenfunctor} and~\ref{prop-rightquillenfunctor} are saying that the Quillen adjunction exchanges bifibrant objects between the two model structures.

Now that we've established that the dualizing complex induces a Quillen adjunction, we turn to showing that it is a Quillen equivalence. This requires that we understand when the unit and counit of the adjunction are weak equivalences.

\begin{proposition}\label{prop-weak-equivvs}
Let $\mathbb{X}$ be a semiseparated Noetherian scheme admitting a dualizing complex $D$.
\begin{enumerate}
\item For all complexes $F$ of flat sheaves, the unit $\eta_F : F \xrightarrow{} \sheafhom_{qc}(D,D\tensor F)$ is a weak equivalence in $\mathfrak{M}^{flat}_{\textnormal{ctr}}$ if and only if it is a pure homology isomorphism. Furthermore, if $F$ is a complex of flat-cotorsion sheaves, this happens if and only if $\eta_F$ is a homotopy equivalence.
\item For all complexes $E$ of injective sheaves, the counit $\epsilon_E : D \tensor \sheafhom_{qc}(D,E) \xrightarrow{} E$ is a weak equivalence in $\mathfrak{M}^{inj}_{\textnormal{co}}$ if and only if $\epsilon_E$ is a homotopy equivalence.
\end{enumerate}
In the affine case $\mathbb{X}= Spec(R)$ of a commutative Noetherian ring $R$ admitting a dualizing complex $D$ we also have
\begin{enumerate}
 \setcounter{enumi}{2}
 \item  For all complexes $P$ of projective modules, not only does (1) apply, but we also have the following:
 Let $\bar{\eta}_P$ be any lift as shown where the bottom row is a cofibrant replacement sequence in $\mathfrak{M}^{proj}_{\textnormal{ctr}}$, that is, a short exact sequence with $dw\widetilde{P}\big (\homcomplex_R(D, D \otimes P)\big ) \in \dwclass{P}$ and $W \in \class{W}_{\textnormal{ctr}}$.
$$\begin{tikzcd}
&  & P \arrow["\eta_P"]{d} \arrow[dashed, "\bar{\eta}_P" ']{dl} \\
W \arrow[tail]{r}  & dw\widetilde{P}\big (\homcomplex_R(D, D \otimes P)\big )  \arrow[two heads]{r} & \homcomplex_R(D, D \otimes P)
\end{tikzcd}$$
Then $\eta_P$ is a weak equivalence in both $\mathfrak{M}^{proj}_{\textnormal{ctr}}$ and $\mathfrak{M}^{flat}_{\textnormal{ctr}}$ if and only if $\bar{\eta}_P$ is a weak equivalence, and this happens if and only if $\bar{\eta}_P$ is a homotopy equivalence.
\end{enumerate}
\end{proposition}

\begin{proof}
Proof of (1). Let $F$ be a complex of flat sheaves. Write $$H := \sheafhom_{qc}(D,D\tensor F),$$ so that the unit is $\eta_F : F \xrightarrow{} H$.
By a pure homology isomorphism we mean to show that $X\tensor F \xrightarrow{X \tensor \eta_F} X\tensor H$ is a homology isomorphism for all complexes of sheaves $X$.
Start by using the model structure $\mathfrak{M}^{flat}_{\textnormal{ctr}}$ to factor the unit as $\eta_F = pi$ where $i$ is a cofibration and $p$ is a trivial fibration. So then $Q := \cok{i}$ is a complex of flats and $C := \ker{p} \in \class{W}_{\textnormal{ctr}}\cap\dwclass{C}$. We have a commutative diagram with short exact sequences as shown.
$$\begin{tikzcd}
F  \arrow[rd, "i" ', tail]  \arrow[rr, "\eta_F"] &   & H  \\
& Z \arrow[ru, "p" ', two heads] \arrow[rd, two heads] & \\
C \arrow[ru, tail] & & Q \\
\end{tikzcd}$$
Since $Z$ is an extension of two complexes of flats, we have $Z \in \dwclass{F}$ too. By the hereditary condition it follows that $C$ too is a complex of flats. Therefore, $C$ is in the core $\dwclass{F}\cap\class{W}_{\textnormal{ctr}}\cap\dwclass{C}$, and so must be a contractible complex of flat-cotorsion sheaves. Note that the short exact sequences are degreewise pure. So applying, for any complex $X$, the functor $X \tensor -$ will yield a similar commutative diagram with the two short exact sequences fixed. Note that $X \tensor p$ is a homology isomorphism since $X \tensor C$ must be exact. By the 2 out of 3 property, $X \tensor \eta_F$ will be a homology isomorphism if and only if $X \tensor i$ is a homology isomorphism. This happens if and only if $X \tensor Q$ is exact for all $X$ and this happens if and only if $Q \in \tilclass{F}$; see~\cite[Prop.~3.4(iv)]{murfet-thesis}. So, $X \tensor \eta_F$ will be a homology isomorphism if and only if $i$ is a trivial cofibration which in turn is equivalent to $\eta_F$ being a weak equivalence.

Now furthermore, suppose $F$ is a complex of flat-cotorsion sheaves. Then by Propositions~\ref{prop-quillenfunctor} and~\ref{prop-rightquillenfunctor} we get that $H$ too is a complex of flat-cotorsion sheaves. In the argument above, the hereditary condition implies that $C, Z,$ and $Q$ all must be complexes of flat-cotorsion sheaves. In particular, $Q \in \dwclass{F}\cap\class{W}_{\textnormal{ctr}}\cap\dwclass{C}$. So both $C$ and $Q$ are contractible complexes of flat-cotorsion sheaves. It means $i$ and $p$, and therefore $\eta_F$ are all homotopy equivalences. This is an if and only if: No matter what, $p$ is a homotopy equivalence. Since we know $i$ is a degreewise split monic, it is a homotopy equivalence if and only if its cokernel is contractible.

The proof of (2) is similar. Let $E$ be a complex of injective sheaves and set $T = D \tensor \sheafhom_{qc}(D,E)$. We note that $T$ is also a complex of injective sheaves. Factor the counit $\epsilon_E : T \xrightarrow{} E$ in $\mathfrak{M}^{inj}_{\textnormal{co}}$ as shown
$$\begin{tikzcd}
T  \arrow[rd, "j" ', tail]  \arrow[rr, "\epsilon_E"] &   & E  \\
& Y \arrow[ru, "q" ', two heads] \arrow[rd, two heads] & \\
J \arrow[ru, tail] & & W \\
\end{tikzcd}$$
where $J$ is a contractible complex of injectives (thus $q$ is a homotopy equivalence). Then $\epsilon_E$ is a weak equivalence if and only if $W \in \class{W}_\textnormal{co}$. But $Y$ is an extension of complexes of injectives, so the hereditary condition implies $W$ is too. Thus $\epsilon_E$ is a weak equivalence if and only if $W \in \class{W}_\textnormal{co} \cap \dwclass{I} = \tilclass{I}$. That is, if and only if $W$ is a contractible complex of injectives. Since $j$ is a degreewise split monic, it is a homotopy equivalence if and only if its cokernel is contractible. Since $\class{W}_\textnormal{co}$ contains all contractible complexes, we conclude $\epsilon_E$ is a weak equivalence if and only if it is a homotopy equivalence.

To prove (3), first recall that $\mathfrak{M}^{proj}_{\textnormal{ctr}}$ and $\mathfrak{M}^{flat}_{\textnormal{ctr}}$ have the same trivial objects, therefore same weak  equivalences. Note that the $W$ in the short exact sequence is trivial, so by the 2 out of 3 axiom, $\eta_P$ is a weak equivalence if and only if $\bar{\eta}_P$ is a weak equivalence. Now do an argument similar in spirit to paragraph one, but with $\bar{\eta}_P$ in place of $\eta_F$ and by using the $\mathfrak{M}^{proj}_{\textnormal{ctr}}$ model structure instead of $\mathfrak{M}^{flat}_{\textnormal{ctr}}$. We conclude that $\bar{\eta}_P$ is a weak equivalence if and only if it is a homotopy equivalence.
\end{proof}

\begin{proposition}\label{prop-units-counits-weak-equivalences}
Let $\mathbb{X}$ be a semiseparated Noetherian scheme admitting a dualizing complex $D$.
\begin{enumerate}
\item For all complexes $F$ of flat sheaves, the unit
$$\eta_F : F \xrightarrow{} \sheafhom_{qc}(D, D \otimes F)$$ is a pure homology isomorphism, equivalently, a weak equivalence in $\mathfrak{M}^{flat}_{\textnormal{ctr}}$. Furthermore, if $F$ is a complex of flat-cotorsion sheaves then $\eta_F$ is a chain homotopy equivalence, equivalently, a weak equivalence in $\mathfrak{M}^{flat}_{\textnormal{ctr}}$.
\item For all complexes $E$ of injective sheaves, the counit
$$\epsilon_E : D \otimes \sheafhom_{qc}(D,E) \xrightarrow{} E$$ is a chain homotopy equivalence, equivalently, a weak equivalence in $\mathfrak{M}^{inj}_{\textnormal{co}}$.
\end{enumerate}
In the affine case $\mathbb{X}=Spec(R)$ of a commutative Noetherian ring $R$ admitting a dualizing complex $D$ we also have
\begin{enumerate}
 \setcounter{enumi}{2}
 \item  For all complexes $P$ of projective modules, the unit
$$\eta_P : P \xrightarrow{} \homcomplex_R(D, D \otimes P)$$ is a weak equivalence in both $\mathfrak{M}^{proj}_{\textnormal{ctr}}$ and $\mathfrak{M}^{flat}_{\textnormal{ctr}}$, equivalently, any lift $\bar{\eta}_P$ as in Proposition.~\ref{prop-weak-equivvs} is a chain  homotopy equivalence.
\end{enumerate}
Moreover, the diagram below is a commutative diagram of adjoint equivalences.
\[\begin{tikzpicture}[node distance=3.5cm]
\node (A) {$K(Proj)$};
\node (B) [right of=A] {$\class{D}(Flat)$};
\node (C) [right of=B] {$K(Inj)$};
\node (D) [above of=B] {$K(Flat\text{-}Cot) $};
\draw[->] (A.15) to node[above]{\small $I$} (B.165);

\draw[<-] (A.346) to node [below]{\small $dw\widetilde{P}$} (B.195);
\draw[->] (B.15) to node[above]{\small $D\tensor -$} (C.165);

\draw[<-] (B.346) to node [below]{\small $\sheafhom_{qc}(D,-)$} (C.195);
\draw[->] (D.350) to node[above,midway,sloped,rotate=3]{\small $D\tensor -$}  (C.30);

\draw[->] (C.130) to node[below,midway,sloped]{\small $\sheafhom_{qc}(D,-)$} (D.325);
\draw[<-] (D.223) to node[left]{\small $dw\widetilde{C}$}  (B.137);

\draw[->] (D.310) to node[right]{\small $I$}  (B.50);
\draw[->] (D.210) to node[below,midway,sloped,rotate=3]{\small $dw\widetilde{P}$}  (A);
\draw[<-] (D.190) to node[above,midway,sloped,rotate=3]{\small $dw\widetilde{C}$}  (A.150);
\end{tikzpicture}\]
The commutative triangle on the right holds for the general case of a semiseparated Noetherian scheme $\mathbb{X}$ admitting a dualizing complex $D$.
\end{proposition}

\begin{remark}
Note that the bottom composite is  Iyengar-Krause's initial finding, although they had the larger category $K(Flat)$ in the middle.  We see that $\class{D}(Flat)$ is equivalent to the reflective subcategory $K(Flat\text{-}Cot) \subseteq K(Flat)$. On the other hand, $\class{D}(Flat)$ is equivalent to the coreflective subcategory $K(Proj) \subseteq K(Flat)$.
\end{remark}

\begin{proof}
We learned from Murfet's thesis, see~\cite[Theorem~8.4]{murfet-thesis}, that
$$\class{D}(Flat) \xrightarrow{D\tensor -} K(Inj)$$ is an equivalence of categories, with adjoint inverse $\sheafhom_{qc}(D,-)$.
The model structure $\mathfrak{M}^{flat}_{\textnormal{ctr}} = (\dwclass{F},  \class{W}_{\textnormal{ctr}},\dwclass{C})$ of Theorem~\ref{them-contra-model-scheme} is such that the inclusion $$I : K(Flat\text{-}Cot) \hookrightarrow \class{D}(Flat)$$ is an equivalence of categories whose inverse is given by fibrant replacement. Since fibrant replacement of an object is calculated precisely by taking a special $\dwclass{C}$-preenvelope, we will denote this functor by $$dw\widetilde{C} : \class{D}(Flat) \xrightarrow{} K(Flat\text{-}Cot).$$
The image of the functor $\sheafhom_{qc}(D,-)$ is contained in $K(Flat\text{-}Cot)$, by Proposition~\ref{prop-rightquillenfunctor}(1), so we have a functor $$\sheafhom_{qc}(D,-) : K(Inj) \xrightarrow{} K(Flat\text{-}Cot).$$ Evidently, its left adjoint is $D \tensor -$, and we conclude that the right triangle of the diagram  is a commutative diagram of adjunctions.
It follows that $$D\tensor - : K(Flat\text{-}Cot) \xrightarrow{} K(Inj)$$ is an adjoint equivalence of categories with inverse $$\sheafhom_{qc}(D,-) : K(Inj) \xrightarrow{} K(Flat\text{-}Cot).$$ In particular, the unit and counit of the adjunction are isomorphisms, which translates to the statements in (2) and the second part of (1); and using Proposition~\ref{prop-weak-equivvs} to infer we have weak equivalences.

But we need to prove the first statement of (1). Given any complex $F$ of flat sheaves, use the model structure $\mathfrak{M}^{flat}_{\textnormal{ctr}}$ to write a short exact sequence $$0 \xrightarrow{} F \xrightarrow{j} C \xrightarrow{} \tilde{F} \xrightarrow{} 0$$
where $\tilde{F} \in \tilclass{F}$ and $C \in \dwclass{F}\cap\dwclass{C}$. It is degreewise pure, so
$$0 \xrightarrow{} D \tensor F \xrightarrow{D \tensor j} D \tensor C \xrightarrow{} D \tensor \tilde{F} \xrightarrow{} 0$$
is still exact. In fact, it is degreewise split with injective components, and $D \tensor \tilde{F}$ is a contractible complex of injectives, by Proposition~\ref{prop-quillenfunctor}. Therefore, applying $\sheafhom_{qc}(D,-)$ we get another short exact sequence
$$0 \xrightarrow{} \sheafhom_{qc}(D,D \tensor F) \xrightarrow{\hat{j}} \sheafhom_{qc}(D,D \tensor C) \xrightarrow{}\sheafhom_{qc}(D, D \tensor \tilde{F}) \xrightarrow{} 0$$
and $\sheafhom_{qc}(D, D \tensor \tilde{F})$ is contractible with flat-cotorsion components, by Proposition~\ref{prop-rightquillenfunctor}.
We get a commutative diagram with short exact sequences for rows:
$$\begin{CD}
      F    @>j>>   C    @>>>    \tilde{F}    \\
  @V \eta_F VV     @V \eta_C VV      @V \eta_{\tilde{F}} VV   \\
  \sheafhom_{qc}(D,D \tensor F)  @>\hat{j} >>   \sheafhom_{qc}(D,D \tensor C)   @>>>    \sheafhom_{qc}(D, D \tensor \tilde{F})   \\
    \end{CD}$$
Now $j$ and $\hat{j}$ are each weak equivalences in $\mathfrak{M}^{flat}_{\textnormal{ctr}}$ because each is monic with cokernel in $\tilclass{F} \subseteq \class{W}_{ctr}$. Now $\eta_C$ is also a weak equivalence by what we just proved above. So it follows from the 2 out of 3 axiom that $\eta_F$ is also a weak equivalence in $\mathfrak{M}^{flat}_{\textnormal{ctr}}$.

Let us turn to (3), the affine case. First, consider the identity functor $\chqcox \xrightarrow{1} \chqcox$. It may be thought of as a Quillen adjunction from $\mathfrak{M}^{proj}_{\textnormal{ctr}}$ to $\mathfrak{M}^{flat}_{\textnormal{ctr}}$. In doing so, the derived functors of the identity provide, when passing to the level of bifibrant objects, the equivalence on the hypotenuse of the left triangle in the diagram. This leads us to the commutative diagram of adjunctions. As in the Remark, the bottom composite is essentially Iyenar-Krause's original equivalence from~\cite{iyengar-krause}. It means that the unit of their adjunction, which in our notation is exactly $\bar{\eta}_P$, is a homotopy equivalence. So by
Proposition~\ref{prop-weak-equivvs}(3) we conclude that the unit $\eta_P$ is a weak equivalence.
\end{proof}

Now we get to our stated goal of lifting the equivalence to the full chain complex category $\chqcox$.

\begin{theorem}\label{them-Quillen-Krause-Iyengar-Murfet}
Let $\mathbb{X}$ be a semiseparated Noetherian scheme with dualizing complex $D$. Then the functor $D \tensor - : \chqcox \xrightarrow{} \chqcox$ is a Quillen equivalence from the flat model structure for the contraderived category to the injective model structure for the coderived category.
In particular, the derived adjunction $(D \tensor Q(-) , \sheafhom_{qc}(D, R(-))$ is a triangulated equivalence
$$\textnormal{Ho}(\mathfrak{M}^{flat}_{\textnormal{ctr}}) \xrightarrow{D \tensor Q(-)} \textnormal{Ho}(\mathfrak{M}^{inj}_{\textnormal{co}})$$
which restricts to a triangulated equivalence $$K(Flat\text{-}Cot) \xrightarrow{D \tensor -} K(Inj)$$  with inverse $\sheafhom_{qc}(D,-)$.

In the affine case $\mathbb{X}=Spec(R)$ of a commutative Noetherian ring $R$ admitting a dualizing complex $D$, then the functor $D \tensor_R - : \ch \xrightarrow{} \ch$ is also a Quillen equivalence from the projective model structure for the contraderived category to the injective model structure for the coderived category.
In this case it restricts to the Krause-Iyengar triangulated equivalence
$$K(Proj) \xrightarrow{D \tensor -} K(Inj)$$ with inverse $dw\widetilde{P} \circ \homcomplex_R(D,-)$.
\end{theorem}

\begin{proof}
Referring to~\cite[Proposition~1.3.13]{hovey-model-categories}(b), given a known Quillen adjunction $(F,U)$ it will be a Quillen equivalence if (i) for each cofibrant $X$ in the domain category, a natural map $X \xrightarrow{} UFX \xrightarrow{} URFX$ is a weak equivalence, and, (ii) for each fibrant $X$ in the codomain category, a natural map $FQUX \xrightarrow{} FUX \xrightarrow{} X$ is a weak equivalence. Applying this to the Quillen adjunction $(D\tensor -, \sheafhom_{qc}(D,-))$ and using the nice properties of Propositions~\ref{prop-quillenfunctor} and~\ref{prop-rightquillenfunctor} it translates to (i) For all complexes $F$ of flat sheaves, the unit
$$\eta_F : F \xrightarrow{} \sheafhom_{qc}(D, D \otimes F)$$ is a weak equivalence in $\mathfrak{M}^{flat}_{\textnormal{ctr}}$, and, (ii) For all complexes $E$ of injective sheaves, the counit
$$\epsilon_E : D \otimes \sheafhom_{qc}(D,E) \xrightarrow{} E$$ is a weak equivalence in $\mathfrak{M}^{inj}_{\textnormal{co}}$. This is what we showed in the previous propositions.

For the affine case, we still need to verify (i) and (ii) from~\cite[Proposition~1.3.13]{hovey-model-categories}(b). The translation of (i) in this case is this: For all complexes $P$ of projectives, the unit $P \xrightarrow{\eta_P} \homcomplex_R(D, D \otimes P)$ is a weak equivalence in $\mathfrak{M}^{proj}_{\textnormal{ctr}}$. This is true by Proposition~\ref{prop-units-counits-weak-equivalences}(3). Next, (ii) says: For all complexes of injective $R$-modules $E$, the map
$$D \otimes dw\widetilde{P}\big(\homcomplex_R(D,E)\big)\xrightarrow{D \tensor q_{\homcomplex_R(D,E)}}  D \otimes \homcomplex_R(D,E) \xrightarrow{\epsilon_E} E$$ is a weak equivalence in $\mathfrak{M}^{inj}_{\textnormal{co}}$. But this map is the counit of the composed adjunctions along the base of the commutative triangle in Proposition~\ref{prop-units-counits-weak-equivalences}. That is, it is the counit of the Iyengar-Krause equivalence and hence it a homotopy equivalence. The same exact type of argument that we gave in the proof of Proposition~\ref{prop-weak-equivvs}(2) shows that it is a weak equivalence in $\mathfrak{M}^{inj}_{\textnormal{co}}$.
\end{proof}

\section{Quillen equivalence of Gorenstein flats and Gorenstein injectives}\label{sec-G-equivalence}

We continue to let $\mathbb{X}$ be a semiseparated Noetherian scheme admitting a dualizing complex $D$. The goal here is to show that the model structure for $\textnormal{St}(G\textnormal{-}Flat\textnormal{-}Cot)$, the stable category of Gorenstein flat-cotorsion sheaves, is Quillen equivalent to the model structure for $\textnormal{St}(G\textnormal{-}Inj)$, the stable category of Gorenstein injective sheaves. Recall that two model structures are called \emph{Quillen equivalent} when there is a zig-zag of Quillen equivalences between them. For example, see~\cite[Section~2]{dugger-shipley}.

We will zig-zag through the F-totally acyclic flat model structure on $\chqcox$ and the totally acyclic injective model structures on $\chqcox$. Thanks to the work of Murfet and Salarian~\cite[Prop.~6.1]{murfet-salarian} it is easy to show that $D\tensor -$ is a Quillen equivalence from the former to the latter.

\begin{theorem}\label{them-Murfet-Salarian}
Let $\mathbb{X}$ be a semiseparated Noetherian scheme  with dualizing complex $D$. Then the functor $D \tensor - : \chqcox \xrightarrow{} \chqcox$ is a Quillen equivalence from $\mathfrak{M}^{flat}_{\textnormal{F-to}}$, the F-totally acyclic flat model structure to $\mathfrak{M}^{inj}_{\textnormal{to}}$, the totally acyclic injective model structure.
In particular, the derived adjunction $(D \tensor Q(-) , \sheafhom_{qc}(D, R(-))$ is a triangulated equivalence
$$\textnormal{Ho}(\mathfrak{M}^{flat}_{\textnormal{F-to}}) \xrightarrow{D \tensor Q(-)} \textnormal{Ho}(\mathfrak{M}^{inj}_{\textnormal{to}})$$
which restricts to a triangulated equivalence $$K_{F\textnormal{-}to}(Flat\text{-}Cot) \xrightarrow{D \tensor -} K_{to}(Inj)$$  with inverse $\sheafhom_{qc}(D,-)$.

In the affine case $\mathbb{X}=Spec(R)$ of a commutative Noetherian ring $R$ admitting a dualizing complex $D$, then the functor $D \tensor_R - : \ch \xrightarrow{} \ch$ is also a Quillen equivalence from $\mathfrak{M}^{proj}_{\textnormal{to}}$, the totally acyclic projective model structure, to $\mathfrak{M}^{inj}_{\textnormal{to}}$, the totally acyclic injective model structure.
In this case it restricts to the triangulated equivalence
$$K_{to}(Proj) \xrightarrow{D \tensor -} K_{to}(Inj)$$ with inverse again $dw\widetilde{P} \circ \homcomplex_R(D,-)$.
In fact,  the commutative diagram of adjoint equivalences from Proposition~\ref{prop-units-counits-weak-equivalences} restricts to the same diagram but with
$K(Proj)$, $\class{D}(Flat)$, $K(Flat\text{-}Cot)$ and $ K(Inj)$ each substituted by $K_{to}(Proj)$, $\class{D}_{F\textnormal{-}to}(Flat)$, $K_{F\textnormal{-}to}(Flat\text{-}Cot)$ and $ K_{to}(Inj)$.
\end{theorem}

\begin{proof}
We are to again show that $D \tensor -$ preserves cofibrations and trivial cofibrations. If $i$ is a cofibration in the F-totally acyclic flat model structure then we have a short exact sequence
$$0 \xrightarrow{} X \xrightarrow{i} Y \xrightarrow{} F \xrightarrow{} 0$$
where $F$ is an F-totally acyclic complex of flat sheaves. Then $i$ is also automatically a cofibration in $\mathfrak{M}^{flat}_{\textnormal{ctr}}$, and by Proposition~\ref{prop-quillenfunctor} we know that applying $D\tensor -$ will produce another short exact sequence. This is all we need to have a cofibration in $\mathfrak{M}^{inj}_{\textnormal{to}}$. In the case that $i$ is a trivial cofibration, then $F \in \tilclass{F}$, and we again already know from Proposition~\ref{prop-quillenfunctor}  that $D\tensor F$ is a contractible complex, hence in $\class{W}_{\textnormal{to}}$, the class of trivial objects in $\mathfrak{M}^{inj}_{\textnormal{to}}$. Therefore $D \tensor -$ is a left Quillen functor.

Now it is proved in~\cite[Prop.~6.1]{murfet-salarian} that a complex of flat sheaves $F$ is F-totally acyclic if and only if $D \tensor F$ is a totally acyclic complex of injectives. Notice then that if $E$ is a totally acyclic complex of injectives, this implies $\sheafhom_{qc}(D,E)$ is an F-totally acyclic complex of flat sheaves. Indeed we already know that $\sheafhom_{qc}(D,E)$ is a complex of flat sheaves, and by Proposition~\ref{prop-units-counits-weak-equivalences}, the counit $$\epsilon_E : D \otimes \sheafhom_{qc}(D,E) \xrightarrow{} E$$ is a chain homotopy equivalence.
Thus $D \otimes \sheafhom_{qc}(D,E)$ would also be a totally acyclic complex of injectives and so the result~\cite[Prop.~6.1]{murfet-salarian} implies that  $\sheafhom_{qc}(D,E)$ is F-totally acyclic.

This makes checking~\cite[1.3.13]{hovey-model-categories}(b) trivial; they reduce to showing that the unit $\eta_F$ is a weak equivalence whenever $F$ is F-totally acyclic and that the counit $\epsilon_E$ is a weak equivalence whenever $E$ is totally acyclic. But by modifying the statements of Proposition~\ref{prop-weak-equivvs} and Proposition~\ref{prop-units-counits-weak-equivalences}, we get the wanted weak equivalences in the same way we did there.

Turning to the affine case, we point out that (by~\cite[Theorem~6.8/6.9 ]{bravo-gillespie-hovey}) there is a cofibrantly generated abelian model structure $\mathfrak{M}^{proj}_{\textnormal{to}} = (\toclass{P}, \class{V}_{\textnormal{to}}, All)$ on $\ch$, whose homotopy category recovers  $K_{to}(Proj)$, the chain homotopy category of all totally acyclic complexes of projectives. Let us first see that the triangular commutative diagram of adjoint equivalences from Proposition~\ref{prop-units-counits-weak-equivalences} restricts to a commutative diagram as we have claimed. Because of the existence of a dualizing complex, we know from~\cite[Lemma~1.7]{jorgensen-tate} that a complex of projectives $P$ is totally acyclic if and only if it is F-totally acyclic. So the inclusion $I : K(Proj) \xrightarrow{} \class{D}(Flat)$ restricts to another inclusion $I : K_{to}(Proj) \xrightarrow{} \class{D}_{F\textnormal{-}to}(Flat)$. As for the inverse functor $dw\widetilde{P}$, let $F \in \class{D}_{F\textnormal{-}to}(Flat)$. To compute $dw\widetilde{P}(F)$ we write a short exact sequence $$0 \xrightarrow{} W \xrightarrow{}  dw\widetilde{P}(F) \xrightarrow{} F  \xrightarrow{} 0$$
expressing $dw\widetilde{P}(F)$ as a special $\dwclass{P}$-precover of $F$. Then note $W \in \rightperp{\dwclass{P}} \cap\dwclass{F}$, which by Neeman's result from~\cite{neeman-flat} implies $W$ is a pure acyclic complex of flat modules. So $W$ and $F$ are each F-totally acyclic, which means $dw\widetilde{P}(F)$ must also be F-totally acyclic. But being a complex of projectives,  it again  means $dw\widetilde{P}(F)$ is a totally acyclic complex of projectives. Similar arguments show that the functor $dw\widetilde{C}$ also preserves F-total acyclicity.

Finally, again using~\cite[1.3.13]{hovey-model-categories}(b)we verify: (i) For all totally acyclic complexes of projectives $P$, the unit $P \xrightarrow{\eta_P} \homcomplex_R(D, D \otimes P)$ is a weak equivalence in $\mathfrak{M}^{proj}_{\textnormal{to}}$. But this is true by making minor modifications to Propositions~\ref{prop-units-counits-weak-equivalences}(3) and~\ref{prop-weak-equivvs}(3). Next, we must verify: (ii) For all totally acyclic complexes of injective $R$-modules $E$, the map
$$D \otimes dw\widetilde{P}\big(\homcomplex_R(D,E)\big)\xrightarrow{D \tensor q_{\homcomplex_R(D,E)}}  D \otimes \homcomplex_R(D,E) \xrightarrow{\epsilon_E} E$$ is a weak equivalence in $\mathfrak{M}^{inj}_{\textnormal{to}}$. But this map is the counit of the composed equivalences along the base of the (newly modified) commutative triangle from Proposition~\ref{prop-units-counits-weak-equivalences}. So it is a homotopy equivalence and therefore a weak equivalence in $\mathfrak{M}^{inj}_{\textnormal{to}}$.

\end{proof}

 Now putting together everything we have proved in this paper we may conclude that we have proved the following theorem.

\begin{theorem}\label{them-G-flat-G-inj}
Let $\mathbb{X}$ be a semiseparated Noetherian scheme  with dualizing complex $D$. Then $$\mathfrak{M}^{flat}_{\textnormal{G}} = (\class{GF}, \class{V}, \class{C}),$$
the Gorenstein flat model structure on $\qcox$, is Quillen equivalent to
$$\mathfrak{M}^{inj}_{\textnormal{G}} = (All, \class{W}, \class{GI}),$$
the Gorenstein injective model structure on $\qcox$, via a zig-zag of Quillen equivalences as shown:
$$ \begin{tikzcd}
\mathfrak{M}^{flat}_{\textnormal{F-to}}  \arrow[r, "D\tensor -", yshift=0.7ex] \arrow[from=r, "\sheafhom_{qc}", yshift=-0.7ex] \arrow[from=d, "S^0" ', xshift=0.7ex] \arrow[d, "G" ', xshift=-0.7ex]
& \mathfrak{M}^{inj}_{\textnormal{to}}  \arrow[d, "Z_0", xshift=0.7ex] \arrow[from=d, "S^0", xshift=-0.7ex] \\
\mathfrak{M}^{flat}_{\textnormal{G}}
&\mathfrak{M}^{inj}_{\textnormal{G}}  \end{tikzcd}$$
Here, $(D\tensor -, \sheafhom_{qc}(D,-))$ is the Quillen equivalence of Theorem~\ref{them-Murfet-Salarian}, $(G,S^0)$ is the Quillen equivalence of Theorem~\ref{them-G-Quillen-equiv}, and $(S^0,Z_0)$ is the Quillen equivalence of Theorem~\ref{them-G-inj}.
\end{theorem}

\section{Injective and projective Tate cohomology}

We learned in~\cite{krause-stable derived cat of a Noetherian scheme} of the (injective) Tate cohomology theory that can be associated to any noetherian Grothendieck category with compactly generated derived category. This applies in particular to $\qcox$ for a semiseparated Noetherian scheme $\mathbb{X}$. Given two sheaves $A$ and $B$, and any integer $n\in\Z$, we get Tate cohomology groups which are denoted by $\widehat{\Ext}^n_{\mathbb{X}}(A,B)$. Here we will call these the \emph{injective Tate cohomology groups} and denote them by ${}_\textnormal{i}\widehat{\Ext}^n_{\mathbb{X}}(A,B)$ to distinguish them from the \emph{projective  Tate cohomology groups} which we will define below and denote by ${}_\textnormal{p}\widehat{\Ext}^n_{\mathbb{X}}(A,B)$. We will show that our projective Tate cohomology is correct in the affine case --- when $R$ is a commutative Noetherian ring with a dualizing complex it agrees with~\cite{jorgensen-tate}. Asadollahi and Salarian have a very similar definition in~\cite{asad-salarian}. As they noticed, their definition works best when the sheaf on the right side is cotorsion. Our definition always introduces a cotorsion approximation on the right hand side. This is in line with what we keeping seeing; it is the Gorenstein flat-cotorsion sheaves which are a good non-affine replacement of Gorenstein projective modules.

\subsection{Injective Tate cohomology} In~\cite{gillespie-canonical resolutions} we find what essentially amounts to a general theory of Tate cohomology that can be attached to any hereditary abelian (or exact) model structure $\mathfrak{M} =(\class{Q},\class{W},\class{R})$. The so-called \emph{canonical resolutions and coresolutions} arising from $\mathfrak{M}$ lead to the construction of abelian groups denoted $\Ext^n_{\textnormal{Ho}(\mathfrak{M})}(A,B)$, which are generalized Tate cohomology groups. For instance we have the following.

\begin{proposition}\label{prop-Tate-G}
Assume $\mathbb{X}$ is a semiseparated Noetherian scheme and let
$$\mathfrak{M}^{inj}_{\textnormal{G}} = (All, \class{W}, \class{GI})$$ be the Gorenstein injective model structure on $\qcox$.
Then a canonical coresolution of a sheaf yields a complete injective resolution of that sheaf. Therefore, for all integers $n$, $$\Ext^n_{\textnormal{Ho}(\mathfrak{M}^{inj}_{\textnormal{G}})}(A,B) \cong {}_\textnormal{i}\widehat{\Ext}^n_{\mathbb{X}}(A,B),$$
and we also have ${}_\textnormal{i}\widehat{\Ext}^n_{\mathbb{X}}(A,B) \cong H^n[\Hom_{\mathbb{X}}(W,RB)]$ where $RB$ is a Gorenstein injective approximation (special pre-envelope) of the sheaf $B$ and $W$ may be any full trivial resolution of the sheaf $A$.
\end{proposition}

We note that a \textbf{full trivial resolution} of an object $A$, in the sense of~\cite[Def.~6.1]{gillespie-canonical resolutions}, is any exact complex $W$ with $A = Z_0W$ and each $W_n \in \class{W}$.

\begin{proof}
Let $B \hookrightarrow RB$ be a Gorenstein injective approximation (special pre-envelope) of the sheaf $B$. A canonical coresolution of $B$, in the sense of~\cite[Def.~1.1]{gillespie-canonical resolutions}, yields precisely a totally acyclic complex of injectives $W^{RB}$ with $RB =Z^0(W^{RB})$. Then the obvious inclusion $B = S^0(B) \hookrightarrow W^{RB}$ is what is often called a \emph{complete injective resolution} of $B$. (See also, \cite[Dual of Them.~9.3/9.4]{gillespie-canonical resolutions}.) Referring to~\cite[Def.~7.5]{krause-stable derived cat of a Noetherian scheme} we have
$$\widehat{\Ext}^n_{\mathbb{X}}(A,B) := H^n[\Hom_{\mathbb{X}}(A,W^{RB})].$$
On the other hand, by~\cite{gillespie-canonical resolutions} (Definition~6.4 and Theorem~6.5(3)) we have a natural isomorphism
$$\Ext^n_{\textnormal{Ho}(\mathfrak{M}^{inj}_{\textnormal{G}})}(A,B) \cong H^n[\Hom_{\mathbb{X}}(A,W^{RB})].$$
Moreover, the class $\class{C}$ throughout~\cite[Section~9]{gillespie-canonical resolutions} coincides with the class of all full trivial resolutions. So now the last statement follows from~\cite[Corollary~9.6]{gillespie-canonical resolutions}.
\end{proof}

\begin{remark}
Note that other general statements from~\cite{gillespie-canonical resolutions}, applied to $\mathfrak{M}^{inj}_{\textnormal{G}}$, correspond to known statements about (injective) Tate cohomology. For example, \cite[Corollary~6.6]{gillespie-canonical resolutions} corresponds to~\cite[Prop.~7.7]{krause-stable derived cat of a Noetherian scheme}, while~\cite[Corollary~6.7]{gillespie-canonical resolutions} recovers the long exact Tate cohomology sequence of~\cite[Prop.~7.9]{krause-stable derived cat of a Noetherian scheme}. Also, \cite[Corollary~7.4]{gillespie-canonical resolutions} corresponds to~\cite[Prop.~7.10]{krause-stable derived cat of a Noetherian scheme}.
\end{remark}

\subsection{Projective Tate cohomology}
So now let $\mathbb{X}$ be a semiseparated Noetherian scheme with a dualizing complex $D$, and we will define the \emph{projective Tate cohomology} for all integers $n$ by
$$ {}_\textnormal{p}\widehat{\Ext}^n_{\mathbb{X}}(A,B) := \Ext^n_{\textnormal{Ho}(\mathfrak{M}^{flat}_{\textnormal{G}})}(A,B).$$
Recall that $\mathfrak{M}^{flat}_{\textnormal{G}}$ denotes the Gorenstein flat model structure $\mathfrak{M}^{flat}_{\textnormal{G}} = (\class{GF}, \class{V}, \class{C})$, so that $\class{GF}$ is the class of Gorenstein flat sheaves and $\class{C}$ is the class of cotorsion sheaves.

By a \emph{complete flat resolution} of a sheaf $A$, we mean a short exact sequence
$$0 \xrightarrow{} V \xrightarrow{} F \xrightarrow{} S^0(A) \xrightarrow{} 0$$ where $F \in {}_I\tilclass{F}$ is an $F$-totally acyclic complex of flat sheaves and $V \in  \rightperp{{}_I\tilclass{F}}$.

\begin{proposition}\label{theorem-projective-Tate}
Assume $\mathbb{X}$ is a semiseparated Noetherian scheme with a dualizing complex $D$. The following hold.
\begin{enumerate}
\item With respect to $\mathfrak{M}^{flat}_{\textnormal{G}}$, a canonical resolution of a sheaf $A$ yields a complete flat resolution  $W_{QA} \twoheadrightarrow S^0(A)$.

\item  ${}_\textnormal{p}\widehat{\Ext}^n_{\mathbb{X}}(A,B) \cong H^n[\Hom_{\mathbb{X}}(F,RB)]$ where $RB$ is a cotorsion approximation (special pre-envelope) of the sheaf $B$ and $F \twoheadrightarrow S^0(A)$ is any complete flat resolution of the sheaf $A$.

\item  ${}_\textnormal{p}\widehat{\Ext}^n_{\mathbb{X}}(A,B) \cong H^n[\Hom_{\mathbb{X}}(QA,W^{RB})]$ where $QA$ is a Gorenstein flat approximation (special precover) of the sheaf $A$ and $W^{RB}$ is any canonical coresolution of the sheaf $B$.

\item  Given any short exact sequence $0 \xrightarrow{} A \xrightarrow{} B \xrightarrow{} C \xrightarrow{} 0$ in $\qcox$ and a sheaf $X$ we have a  long exact sequence of abelian group  $$\cdots \xrightarrow{} {}_\textnormal{p}\widehat{\Ext}^{n-1}_{\mathbb{X}}(X,C) \xrightarrow{} {}_\textnormal{p}\widehat{\Ext}^n_{\mathbb{X}}(X,A) \xrightarrow{} {}_\textnormal{p}\widehat{\Ext}^n_{\mathbb{X}}(X,B)$$
$$ \xrightarrow{} {}_\textnormal{p}\widehat{\Ext}^n_{\mathbb{X}}(X,C) \xrightarrow{} {}_\textnormal{p}\widehat{\Ext}^{n+1}_{\mathbb{X}}(X,A) \xrightarrow{} \cdots $$ and similar for the contravariant ${}_\textnormal{p}\widehat{\Ext}^n_{\mathbb{X}}(-,X)$.

\item ${}_\textnormal{p}\widehat{\Ext}^n_{\mathbb{X}}(A,B)$ identifies with the $\Hom$ groups:
$$\textnormal{Ho}(\mathfrak{M}^{flat}_{\textnormal{G}})(\Omega^n A, B) \cong {}_\textnormal{p}\widehat{\Ext}^n_{\mathbb{X}}(A,B) \cong
 \textnormal{Ho}(\mathfrak{M}^{flat}_{\textnormal{G}})(A, \Sigma^n B).$$

 \item For positive integers $n$, ${}_\textnormal{p}\widehat{\Ext}^n_{\mathbb{X}}(A,B)$ identifies with the $\Ext$ group:
$$ {}_\textnormal{p}\widehat{\Ext}^n_{\mathbb{X}}(A,B) \cong \Ext^n_{\mathbb{X}}(QA,RB),$$
where $QA$ is a Gorenstein flat approximation (special precover) of $A$ and $RB$ is a cotorsion approximation (special pre-envelope) of $B$.

\item There are dimension shifting formulas:
$${}_\textnormal{p}\widehat{\Ext}^{n+m}_{\mathbb{X}}(\Sigma^m A,B) \cong
{}_\textnormal{p}\widehat{\Ext}^n_{\mathbb{X}}(A,B) \cong   {}_\textnormal{p}\widehat{\Ext}^{n+m}_{\mathbb{X}}(A, \Omega^m B).$$
%
\end{enumerate}
Moreover, in the affine case of a commutative Noetherian ring $R$ admitting a dualizing complex $D$, then ${}_\textnormal{p}\widehat{\Ext}^n_{\mathbb{X}}(A,B)$ agrees with the usual projective Tate cohomology defined via complete projective resolution of $A$.
\end{proposition}

\begin{proof}
We claim that Becker's  model structure $(\class{C},\class{V},\dgclass{R})$ that appears throughout \cite[Section~9]{gillespie-canonical resolutions} is none other than $\mathfrak{M}^{flat}_{\textnormal{F-to}} = ({}_I\tilclass{F}, \class{V}_{\textnormal{to}}, \dwclass{C})$, the F-totally acyclic flat model structure of Theorem~\ref{them-F-flat}. To see this, we just need to verify that ${}_I\tilclass{F}$, the class of all F-totally acyclic complexes of flats, coincides with $\class{C}$, the class of all acyclic complexes of flats with Gorenstein flat cycles. Certainly, ${}_I\tilclass{F} \subseteq \class{C}$. For the reverse containment, it is enough by show that any $F \in \class{C}$ satisfies $\Hom_{\mathbb{X}}(F,C)$ is acyclic whenever $C$ is a flat-cotorsion sheaf~\cite[Theorem~4.18(ii)]{murfet-salarian}. But any flat-cotorsion sheaf is Gorenstein cotorsion, by~\cite[Lemma~2.3]{cet-G-flat-stable-scheme}. So such a $\Hom_{\mathbb{X}}(F,C)$ is certainly acyclic because $\Ext^1_{\mathbb{X}}(Z_nF,C) = 0$.

So now statement (1) is a special case of~\cite[Theorem~9.3]{gillespie-canonical resolutions} and statement (2) is a special case of~\cite[Corollary~9.4]{gillespie-canonical resolutions}.

Statements (3) and (4) follow from~\cite[Theorem~6.5]{gillespie-canonical resolutions}. The identification of projective Tate cohomology as $\Hom$ groups is~\cite[Corollary~6.6]{gillespie-canonical resolutions}.
The identification as $\Ext$ groups is~\cite[Corollary~7.4]{gillespie-canonical resolutions}. The dimension shifting formulas~\cite[Corollary~6.8]{gillespie-canonical resolutions}.

For the affine case, we see from~\cite[Remark~1.6]{jorgensen-tate} that all flat modules have finite projective dimension whenever $R$ is a commutative Noetherian ring with a dualizing complex. So by combining~\cite[Prop.~5.2]{estrada-gillespie-coherent-schemes} and~\cite[Theorem~6.7]{bravo-gillespie-hovey}, the Gorenstein flat model structure $\mathfrak{M}^{flat}_{\textnormal{G}}$ shares the exact same class $\class{V}$ of trivial objects as the Gorenstein projective model structure $\mathfrak{M}^{proj}_{\textnormal{G}} = (\class{GP}, \class{V},All)$. So now the point is that projective Tate cohomology exists as morphism sets in the common category $\textnormal{Ho}(\mathfrak{M}^{proj}_{\textnormal{G}}) = \textnormal{Ho}(\mathfrak{M}^{flat}_{\textnormal{G}})$.  A canonical resolution with respect to $\mathfrak{M}^{proj}_{\textnormal{G}}$ yields a complete projective resolution in the usual sense, proving the final claim.
 \end{proof}

\subsection{Gorenstein schemes}
We will show now that the previously introduced Tate cohomology functors agree in the case that $\mathbb{X}$ is a semiseparated Gorenstein scheme of finite Krull dimension, (equivalently, admits a dualizing complex~\cite[Ch.V.10]{hartshorne-residues and duality}). This will be a consequence of the fact that over such schemes the class of trivial objects in the two models (the Gorenstein injective and the Gorenstein flat model structure) coincides.
Throughout this section, we fix a semiseparating open affine covering $\mathcal U=\{U_0,\ldots, U_m\}$ of $\mathbb{X}$.
Firstly, we will summarize some properties on Gorenstein schemes.

We recall that a commutative Noetherian ring is called \emph{Iwanaga-Gorenstein} if it has finite self-injective dimension.

A scheme $\mathbb{X}$ is Gorenstein provided that $\mathcal O_{\mathbb{X},x}$ is a Gorenstein ring for every $x\in \mathbb{X}$. If, in addition, $\mathbb{X}$ has finite Krull dimension $d$, then this is equivalent to saying that $\mathcal{O}_{\mathbb{X}}(U)$ is Iwanaga-Gorenstein (or \emph{$d$-Gorenstein}, following \cite[Definition 9.1.9]{enochs-jenda-book}), for each open affine $U$ of $\mathbb{X}$ or, equivalently, for each $U_i \in \mathcal U$.

Given $M\in \qcox$, we will denote by $\mathrm{id}_{\mathbb{X}}(M)$ and $\mathrm{fd}_{\mathbb{X}}(M)$ the injective dimension and the flat dimension of $M$, respectively. It is then clear that $\mathrm{id}_{\mathbb{X}}(M)<\infty$ if and only if $\mathrm{id}_{\mathcal O_{\mathbb{X}}(U_i)}(M(U_i))<\infty$ and $\mathrm{fd}_{\mathbb{X}}(M)<\infty$ if and only if $\mathrm{fd}_{\mathcal O_{\mathbb{X}}(U_i)}(M(U_i))<\infty$, for each $i\in \{0,\ldots,m\}$.

Now, although $\mathbb{X}$ may not have enough projective objects, we still can define $\mathrm{pd}_{\mathbb{X}}(M)$ in terms of vanishing of the Ext functor, i.e. $\mathrm{pd}_{\mathbb{X}}(M)\leq n$ if and only if $\Ext_{\mathbb{X}}^i(M,-)=0$, for $i\geq n+1$.
\begin{lemma}
Let $\mathbb X$ be a semiseparated Gorenstein scheme of finite Krull dimension $d$. The following are equivalent for a sheaf $M$.
\begin{enumerate}
\item $\mathrm{id}_{\mathbb{X}}(M)<\infty$.
\item $\mathrm{fd}_{\mathbb{X}}(M)<\infty$.
\item $\mathrm{pd}_{\mathbb{X}}(M)<\infty$.
\end{enumerate}
\end{lemma}
\begin{proof}
The equivalence $(1)\Leftrightarrow (2)$ follows by the comment above and by \cite[Proposition 9.1.7]{enochs-jenda-book}. Now $(2)\Rightarrow (3)$ follows from \cite[Lemma 6.9]{asad-salarian} where the authors show that every flat sheaf has finite projective dimension $\leq m+d$. Finally, let us assume $(3)$; then, by \cite[Lemma 6.5]{gillespie-quasi-coherent}, given an open affine $U_i$ of $\mathcal U$ and the inclusion $j : U_i \subseteq X$, the direct image functor $j_*$ preserves quasi-coherence and there are natural isomorphisms
$$\Ext^n_{\mathcal{O}_{\mathbb{X}}(U_i)}(M(U_i),G) \cong \Ext^n_{\mathbb{X}}(M,j_*G),$$
for all $n\geq 0$. It follows immediately that $\mathrm{pd}_{\mathcal O_{\mathbb{X}}(U_i)}(M(U_i))<\infty$, for each $i\in \{0,\ldots,m\}$. Now, \cite[Proposition 9.1.7]{enochs-jenda-book} tells us that $\mathrm{fd}_{\mathcal O_{\mathbb{X}}(U_i)}(M(U_i))<\infty$, for each $i\in \{0,\ldots,m\}$. Therefore $(2)$ follows.
\end{proof}

\begin{lemma}\label{lemma-Gorenstein scheme}
Assume $\mathbb X$ is a semiseparated Gorenstein scheme of finite Krull dimension $d$. Then:
\begin{enumerate}
\item Every $dth$ injective cosyzygy of a sheaf is Gorenstein injective.
\item Every $dth$ flat syzygy of a sheaf is Gorenstein flat.
\end{enumerate}
\end{lemma}

\begin{proof}
Let us see $(1)$. By \cite[Theorem 3.6]{estrada-iacob-frontiers-china} the notion of Gorenstein injective can be tested locally on each $U_i\in \class{U}$ and, by hypothesis, each ring $\class{O}_{\mathbb{X}}(U_i)$ is $d$-Gorenstein, so then the result follows from \cite[Theorem 12.3.1]{enochs-jenda-book}. The proof of $(2)$ also follows from \cite[Theorem 12.3.1]{enochs-jenda-book} as the notion of Gorenstein flat is also local by \cite[Theorem 1.6]{cet-G-flat-stable-scheme}.
\end{proof}

\begin{proposition}
Let $\mathbb X$ be a semiseparated Gorenstein scheme of finite Krull dimension. Then the Gorenstein flat model $\mathfrak{M}^{flat}_{\textnormal{G}}$ and the Gorenstein injective model $\mathfrak{M}^{inj}_{\textnormal{G}}$ on $\qcox$ have the same class of trivial objects $\mathcal W$ given by the sheaves $M$ with $\mathrm{id}_{\mathbb{X}}(M)<\infty$ (equivalently, $\mathrm{pd}_{\mathbb{X}}(M)<\infty$ or $\mathrm{fd}_{\mathbb{X}}(M)<\infty$).
\end{proposition}
\begin{proof}
To see that $\mathcal W$ is the class of trivial objects in the Gorenstein injective model $\mathfrak{M}^{inj}_{\textnormal{G}}$, it suffices to show that $\leftperp{\class{GI}}=\class{W}$. Since every $M\in \class{W}$ has finite injective dimension an easy dimension shifting argument  gives immediately that $M\in \leftperp{\class{GI}}$. Conversely, assume that $T\in \leftperp{\class{GI}}$ and let $L$ be any sheaf. We will prove that $\Ext^{d+i}_{\mathbb{X}}(T,L)=0$, where $d=\mathrm{dim}(\mathbb{X})$, for all $i\geq 1$.  By Lemma \ref{lemma-Gorenstein scheme}, the $dth$ injective cosyzygy of $L$, say $G$, is Gorenstein injective, therefore $\Ext_{\mathbb{X}}^{d+i}(T,L)=\Ext_{\mathbb{X}}^{i}(T,G)=0$. Thus $\mathrm{pd}_{\mathbb{X}}(T)\leq d<\infty $, and so $T\in \class{W}$.

Let $(\class{F},\class{C})$ denote the flat cotorsion pair. To see that $\mathcal W$ is the class of trivial objects in the Gorenstein flat model $\mathfrak{M}^{flat}_{\textnormal{G}}$, it suffices to show that $\rightperp{\class{GF}}=\class{W}\cap \class{C}$ (notice that it is immediate that $\class{GF}\cap\class{W}=\class{F}$). Let us consider $H\in \class{W}\cap \class{C}$ and let $G$ be a Gorenstein flat sheaf, thus $G=Z_0 F$, where $F$ is an F-totally acyclic complex of flat sheaves. Since $H\in \class{W}$, it has finite injective dimension, say $\mathrm{id}_{\mathbb{X}}(H)\leq n$. Since $H$ is cotorsion, we can apply dimension shifting to get $\Ext_{\mathbb{X}}^{1}(G,H)=\Ext_{\mathbb{X}}^{n+1}(Z_{-n}F,H)=0 $, i.e. $H\in \rightperp{\class{GF}}$. Conversely, assume $D\in\rightperp{\class{GF}}$ (so in particular $D$ is cotorsion) and let $L$ be any sheaf. As in the previous case, we will prove that $\Ext^{d+i}_{\mathbb{X}}(L,D)=0$, where $d=\mathrm{dim}(\mathbb{X})$, for all $i\geq 1$. By Lemma \ref{lemma-Gorenstein scheme}, the $dth$ flat syzygy of $L$, say $K$, is Gorenstein flat, therefore $\Ext_{\mathbb{X}}^{d+i}(L,D)=\Ext_{\mathbb{X}}^{i}(K,D)=0$. Then, $\mathrm{id}_{\mathbb{X}}(D)\leq d<\infty $, and so $D\in \class{W}$.
\end{proof}

\begin{remark}
A close look at what has been shown above reveals that the Krull dimension $d$ is the common upper bound to $\mathrm{id}_{\mathbb{X}}(M)$, $\mathrm{fd}_{\mathbb{X}}(M)$, and $\mathrm{pd}_{\mathbb{X}}(M)$, whenever any of these dimensions is indeed finite. For $\mathrm{id}_{\mathbb{X}}(M)$ and $\mathrm{fd}_{\mathbb{X}}(M)$, this follows from the corresponding fact for $d$-Gorenstein rings, and for $\mathrm{pd}_{\mathbb{X}}(M)$ it follows from the first paragraph of the proof of the above Proposition. We have shown that $\qcox$ is a \emph{Gorenstein category} in the sense of \cite[Defs.~2.15 and~2.18]{enochs-estrada-grozas} whenever $\mathbb{X}$ is a semiseparated Gorenstein scheme of finite Krull dimension. This extends \cite[Theorem 3.11]{enochs-estrada-grozas} where this is proved for Gorenstein projective schemes of finite Krull dimension.

\end{remark}

Since the two model structures share the same class of trivial objects they are isomorphic representations  of the same homotopy (localization) category. That is,
$\textnormal{Ho}(\mathfrak{M}^{flat}_{\textnormal{G}}) \cong \textnormal{Ho}(\mathfrak{M}^{inj}_{\textnormal{G}})$.
Moreover, the loop $\Omega$ and suspension $\Sigma$ functors agree in the two different model structures, and this is shown directly in~\cite[Appendix~A]{gillespie-canonical resolutions}. So we have proved the following theorem, expressing balance between the (mock) projective Tate cohomology and the injective Tate cohomology.


\begin{theorem}\label{them-mock-tate-inj-tate}
Let $\mathbb X$ be a semiseparated Gorenstein scheme of finite Krull dimension, and $A,B \in \qcox$. Then the (mock) projective Tate cohomolgoy and the injective Tate cohomology agree:
$${}_\textnormal{p}\widehat{\Ext}^n_{\mathbb{X}}(A,B) \cong {}_\textnormal{i}\widehat{\Ext}^n_{\mathbb{X}}(A,B).$$
\end{theorem}

\providecommand{\bysame}{\leavevmode\hbox to3em{\hrulefill}\thinspace}
\providecommand{\MR}{\relax\ifhmode\unskip\space\fi MR }
\providecommand{\MRhref}[2]{%
  \href{http://www.ams.org/mathscinet-getitem?mr=#1}{#2}
}
\providecommand{\href}[2]{#2}

\end{document}